\documentclass[11pt]{article}
\usepackage{amssymb,amsmath,graphicx,hyperref}
\usepackage[utf8]{inputenc}

\newtheorem{theorem}{Theorem}
\newtheorem{proposition}{Proposition}
\newtheorem{lemma}{Lemma}

\newtheorem{corollary}{Corollary}
\newtheorem{conjecture}{Conjecture}
\newtheorem{definition}{Definition}
\newtheorem{remark}{Remark}
\newenvironment{proof}{\noindent \emph{Proof. }}{\hfill \hbox{\rlap{$\sqcap$}$\sqcup$}\\}

\title{On Dense Tetrahedra in Binary Sphere Packings}
\author{
Thomas Fernique
\footnote{HSE University, Moscow, Russia}
\and
Daria Pchelina
\footnote{Laboratoire LIP, ENS Lyon \& Inria Lyon, France}
}

\begin{document}
\maketitle

\begin{abstract}
This paper considers the density of tetrahedra arising in a specific decomposition of packings of unequal spheres in $\mathbb{R}^3$.
It aims to extend a bound obtained in 2D in the 1960s by Florian. The focus is on packings of spheres of sizes $1$ and $\sqrt{2}-1$: the small sphere fits exactly into each octahedral hole of a hexagonal close packing of large spheres, yielding a conjecturally maximally dense packing (for these sizes).
The paper slightly improves, by completely different means, the previous best upper bound on the density of such packings.
The proof combines geometric insight with challenging interval arithmetic computations, which may be of independent interest.
\end{abstract}


\noindent The $10$ computer programs accompanying this article can be found there:
\begin{center}
\url{https://github.com/fernique/florian3D}
\end{center}

\section{Introduction}

A {\bf sphere packing} is a set of spheres with disjoint interiors in a given space.
The {\bf density} of a packing is the proportion of space inside these spheres\footnote{Formally, the density is defined as the limit superior when $r\to\infty$ of the proportion of the points within distance $r$ of the origin.}.

The most natural example is spheres in three-dimensional Euclidean space.
It is particularly useful for studying the structure of materials, with spheres representing atoms or nanoparticles.
In this context, the forces of interaction between particles seem to favor denser structures (although the link between forces and density is still mathematically poorly understood).
This has motivated physicists to search for particularly dense packings by numerical simulation \cite{HO11,HST12}.
Physicists are also interested in the case of dimension 2, e.g. when particles are cylinders \cite{LH93,FJFS20}.

Sphere packings in higher dimensions or more exotic spaces are also studied, motivated by the search for good {\em error-correcting codes}.
In a nutshell: the centers of spheres encode packets of information, and any point in a sphere is seen as an erroneous version of the information encoded by the center of that sphere.
Thus, the size of the spheres determines the robustness of the code, and a dense packing allows more packets of information to be encoded in a given volume.
A comprehensive reference on the subject is \cite{CS99}.

The first non-trivial density result is that of unit disk packings in the Euclidean plane.
It was shown that the maximum density was that of unit disks centered on the triangular lattice of size $2$ (disks centered on two vertices of the same triangle are therefore tangent).
A first partial proof was published in 1910 \cite{Thu10}.
A complete proof was given in 1943 \cite{FT43}.
The most elegant proof is probably that provided in 2010 \cite{CW10}: it shows that in any Delaunay triangulation of the centers of a packing, the density of a triangle (i.e. the proportion of the triangle's area covered by disks) is at most the one of a triangle whose three disks are pairwise tangents.
Recall that the {\bf Delaunay triangulation} of a set of points (here the disk centers) is a triangulation such that no point lies inside the circumcircle of any triangle -- it tends to produce triangles that are as regular as possible, avoiding long, skinny ones.
The same bound for the entire packing follows immediately since there exists a packing formed only of such triangles (Fig.~\ref{fig:intro1}).

\begin{figure}[htbp]
\centering
\includegraphics[width=\textwidth]{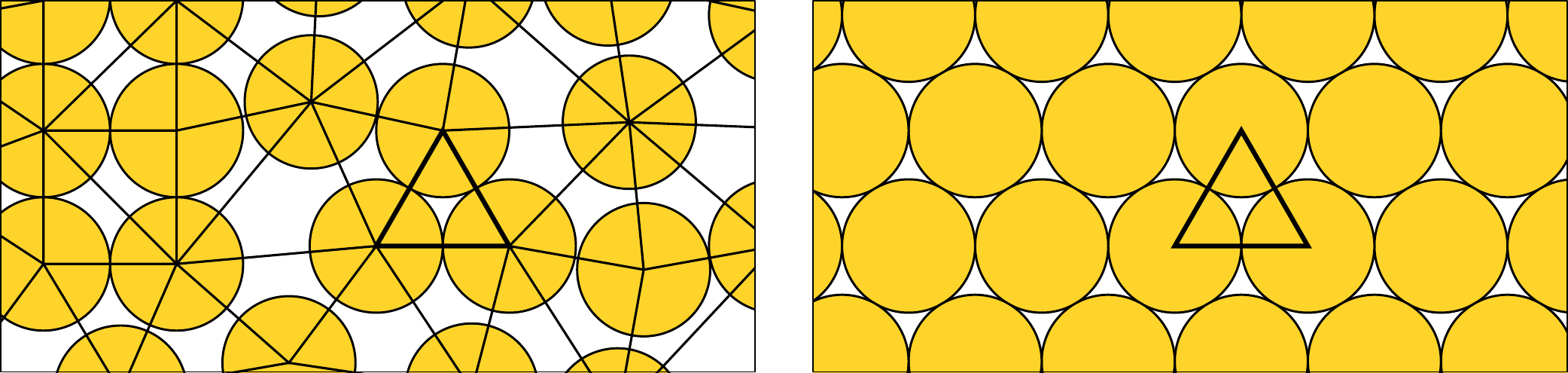}
\caption{
Left: a packing of unit disks in the Euclidean plane and a Delaunay triangulation of the disk centers, with the densest possible triangle emphasized.
Right: a packing made of densest possible triangles only.
}
\label{fig:intro1}
\end{figure}

The case of unit spheres in dimension 3 is much more difficult, and has occupied scientists since at least Kepler's 1610 conjecture, which states that maximum density is reached by the so-called ``cannonball packing'' or {\bf hexagonal close packing} (HCP): the spheres form layers on a triangular grid (as in dimension $2$) and each sphere nestles in the hollow formed by three spheres from the previous layer  (Fig.~\ref{fig:intro2}, left).
Thus, a sphere and the three spheres it touches from the previous layer are centered on a regular tetrahedron.
It was demonstrated in 1958 that this tetrahedron maximizes the density among tetrahedra that can appear in a Delaunay triangulation of the centers of a packing's spheres \cite{Rog58}.
This regular tetrahedron thus plays the same role as the equilateral triangle in the proof of \cite{CW10} mentioned above.
Unfortunately, regular tetrahedra do not tile the Euclidean space of dimension 3.
In particular, in a ``cannonball packing'', the centers of tangent spheres do not always form tetrahedra because, when the spheres of one layer wedge between those of the previous layer, they occupy only half the holes.
Each of the remaining holes becomes the center of an octahedron formed by three spheres of one layer and three spheres of the previous layer (physicists call this an {\em octahedral site}).
Fig.~\ref{fig:intro2}, right, depicts such a packing.
This is sometimes referred to as a {\em frustration} phenomenon: the global optimum (the densest arrangement of spheres over the whole space) is not locally optimal (some spheres could do better if they were not obstructed by others).
Nevertheless, the upper bound on the maximum density of a packing of unit spheres implied by this result remained the best known until the work of Thomas Hales in the late 90s \cite{Hal97}.

\begin{figure}[htbp]
\centering
\includegraphics[width=\textwidth]{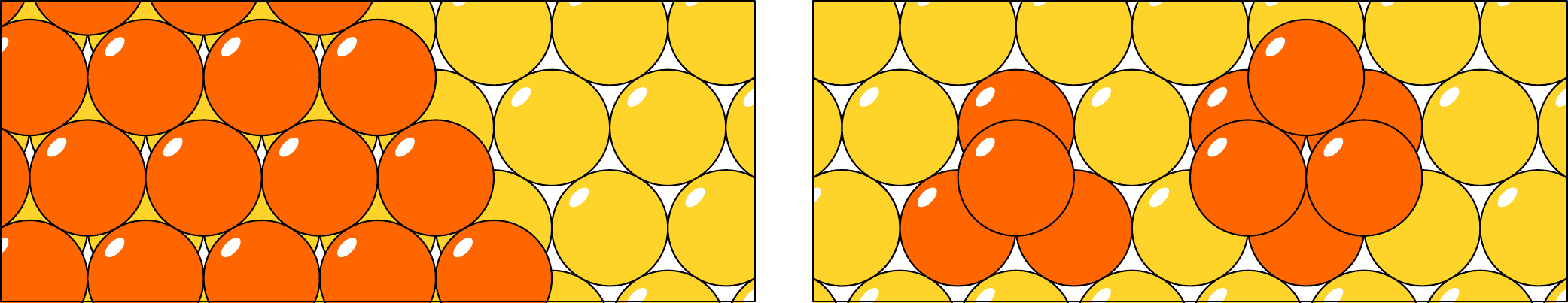}
\caption{
Left: a layer of a hexagonal compact packing (yellow spheres) and some spheres (oranges) of the next layer.
Right: how spheres of consecutive layers form tetrahedra or octahedra.
}
\label{fig:intro2}
\end{figure}

This phenomenon of frustration actually appears as early as dimension 2 for packings of disks of different sizes.
Indeed, it has been proved in \cite{Flo60} that in any {\em FM-triangle decomposition} (a modification of Delaunay triangles introduced in \cite{FM58}) of the centers of a packing of disks of sizes within an interval $[r,1]$, the density of a triangle does not exceed that of a triangle formed by a disk of size $1$ and two disks of size $r$, all three mutually tangent (Fig.~\ref{fig:intro3}, left).
Unfortunately, just as the regular tetrahedron does not tile the space, this triangle does not tile the plane (Fig.~\ref{fig:intro3}, right).
Let us mention that the maximum density was nevertheless characterized for a few specific disk sizes (two sizes in \cite{Hep00,Hep03,Ken04,BF22}, three in \cite{FP23}).

\begin{figure}[htbp]
\centering
\includegraphics[width=0.8\textwidth]{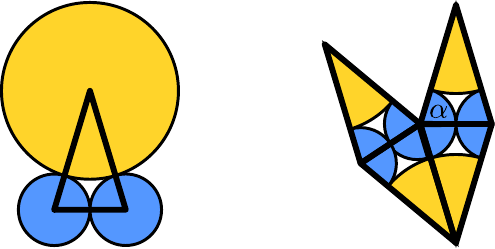}
\caption{
Left: the densest triangle that can appear in a packing of disks of sizes in $[r,1]$.
Right: there is no packing with only such triangles.
Consider indeed the $k$ triangles which share a common small disk: since these triangles must always be paired along their shortest edge, $k$ is even, but $\alpha<\pi/2$ yields $k>4$ while $\alpha>\pi/3$ (because $r<1$) yields $k<6$.
}
\label{fig:intro3}
\end{figure}

What about higher dimension?
To the best of our knowledge, the few known upper bounds have all been obtained using the method introduced by Cohn and Elkies in \cite{CE03} consisting in bounding the density on well-chosen triangles.
Very schematically, this method involves finding a real function $f$ of a variable $x\in\mathbb{R}^n$ that is negative for $|x|\geq 2$ and whose Fourier transform $\widehat{f}$ is positive everywhere; the {\em Poisson summation formula} applied to the centers of the unit spheres of a packing then allows us to use both the fact that these centers are $x\geq 2$ apart (since the spheres do not intersect) and the positivity of $\widehat{f}$ to bound the packing density by $f(0)/\widehat{f}(0)$.
Cohn and Elkies found functions that improved on the best known bounds in dimensions 4 to 24.
In particular, their bounds in dimension 8 and 24 were within a factor less than 1.001 of the conjectured maximum density, which was subsequently proven to be optimal with very ingenious functions \cite{Via17,CKMRV17}.

The method of Cohn and Elkies was extended to the case of several sphere sizes in \cite{Lov14}.
While the results are somewhat disappointing in dimension 2 (Florian's bound does much better), they have no equivalent in higher dimensions.
In particular, the authors found a function allowing them to bound from above by $0.813$ the density of a packing of spheres of size $1$ and $\sqrt{2}-1$.
This particular case is the focus of the work presented in this article.

The case of two spheres of size $1$ and $\sqrt{2}-1$ is indeed particularly interesting, as this is the maximum ratio that allows a small sphere to be inserted into the octahedral sites of an optimal packing of spheres of size $1$ (Fig.~\ref{fig:intro4}).
Moreover, it has been proven in \cite{Fer21} to be the only size ratio that makes it possible to construct a packing whose {\bf contact graph}, whose vertices are the centers of the spheres, connected by an edge when the spheres are tangent, is a ``tetrahedralization'' (that is a generalization of a triangulation -- formally a {\em homogeneous simplicial $3$-complex}).
This is an important point because the cases where a packing of this type exists are the only ones -- with the notable exception of Kepler's conjecture -- for which the maximum density has been characterized.
Intuitively, in such a packing, the spheres arrange themselves locally well, which suggests that it is possible to prove “local optimality implies global optimality” (the emblematic example of which is the packing of unit disks in the plane).
Indeed, this packing is conjectured (e.g. in \cite{HST12}) to maximize the density among packings with spheres of size $1$ and $\sqrt{2}-1$, namely with the following value, slightly less than the above-mentioned upper bound of $0.813$:
\[
\left(\tfrac{5}{3}-\sqrt{2}\right)\pi\approx 0.793.
\]

\begin{figure}[htbp]
\centering
\includegraphics[width=\textwidth]{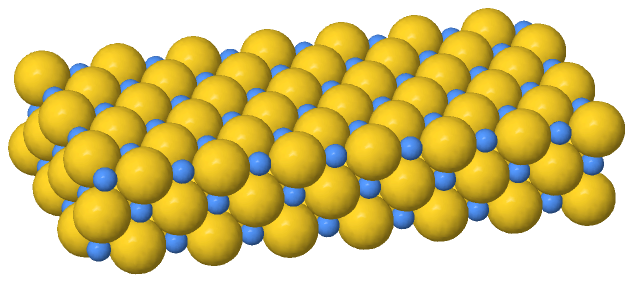}
\caption{
A cannonball packing of unit spheres (yellow) with spheres of size $\sqrt{2}-1$ filling octahedral sites.
}
\label{fig:intro4}
\end{figure}

The aim of this article is to obtain an upper bound on the density of packings by spheres of size $1$ and $\sqrt{2}-1$.
We follow the strategy of \cite{CW10,Rog58,Flo60}: decompose the space into specific tetrahedra -- namely the FM-tetrahedra defined in Section~\ref{sec:settings} -- and bound from above the {\bf tetrahedron density}, defined as the proportion of the volume of the tetrahedron that lies inside the spheres of the packing.
In order to state the main result, we need to introduce two terms.
The {\bf type} of a tetrahedron is the quadruple formed by the radii of the spheres centered on its vertices.
An edge of a tetrahedron is said to be {\bf tight} if the spheres centered at its endpoints are tangent.
Then:
 
\begin{theorem}
\label{th:main}
Let $r:=\sqrt{2}-1$.
The densest FM-tetrahedra of types $1111$, $11rr$ and $1rrr$ have only tight edges.
The densest FM-tetrahedron of type $rrrr$ has $4$ tight edges and two edges that share a vertex and have length
\[
r\sqrt{2\sqrt{6}+6}\approx 1.367.
\]
The densest FM-tetrahedron of type $111r$ has $5$ tight edges and one edge between two large spheres of length
\[
4\sqrt{\frac{2r}{1+2r}}\approx 2.692.
\]
\end{theorem}

\begin{figure}[htb]
\centering
\includegraphics[width=0.9\textwidth]{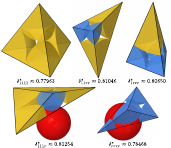}
\caption{
The densest FM-tetrahedra of each type, with a numerical approximation of their density.
Blue spheres have radius $\sqrt{2}-1$, yellow ones radius $1$, and red ones are the support spheres defined in Section~\ref{sec:settings}: they are tangent to the four other spheres).
In the two bottom cases, the upper blue sphere is tangent to the opposite face and the red sphere has the same size.
}
\label{fig:densest_tetrahedra}
\end{figure}

Figure~\ref{fig:densest_tetrahedra} depicts each of these densest FM-tetrahedra and provides a numerical approximation of their densities (which can be computed from the data in the theorem).

\begin{corollary}
A packing of spheres of radii $1$ and $\sqrt{2}-1$ has density at most
\[
\delta^*_{111r}=
0.812542027810834866943600528883352220338748559263354479\ldots 
\]
\end{corollary}

This improves very slightly the value $0.81358\ldots$ obtained in \cite{Lov14} through entirely different means — namely, a semi-definite programming search for a suboptimal function for an extension to several sphere sizes of the Cohn-Elkies method \cite{CE03}.
The geometrical insight into the structure of an optimal packing obtained here could help to add extra constraints to that semi-definite programming approach and may lead to better bounds.

It is still larger than the conjectured maximal density ($0.793$) because the packing depicted in Fig~\ref{fig:intro4} is made of tight tetrahedra of type 1111 and 111r, not of the densest tetrahedron of type 111r, that does not tile the space (again a frustration phenomenon).

Florian's result \cite{Flo60} characterizes the densest triangle in a packing of disks with sizes in the whole interval $[r,1]$.
In contrast, we characterize the densest tetrahedron only for the sizes $1$ and $r=\sqrt{2}-1$.
Actually, as we shall see, the densest tetrahedron is likely to depend on $r$.
Florian's proof, from which it seems natural to draw inspiration here, works in two stages:
\begin{enumerate}
\item 
Show that density is maximized for a triangle with two edges along which the disks touch (tight edges) -- this stage is actually done by Fejes T\'oth and Moln\'ar in \cite{FM58} (we have coined the term FM-tetrahedron from the authors' names).
\item
Study the density as a function of three variables -- the length of the non-tight edge and the two smallest disk radii -- this is what is done in \cite{Flo60}.
\end{enumerate}
However, the task is considerably complicated here due to an elementary fact: while the sum of the angles of a triangle is always $\pi$, that of the solid angles of a tetrahedron is variable.
This fact indeed plays a key role in the first point of the Florian theorem proof, see Fig.~\ref{fig:intro5}.
Here, despite our best efforts, we have only managed to show that density is maximized on a tetrahedron with at least one tight edge.
Hence, to mimic the second stage of Florian's theorem, we need to study a function of 5 variables (the lengths of the other edges).
Moreover, our function is considerably more complicated, as the expression of the density is much simpler for a triangle than for a tetrahedron.
Since Florian's study is already relatively tedious (the 13-page article \cite{Flo60} is essentially made up of calculations that go as far as deriving the studied function 5 times), we suspected that this task could not be carried out by hand.
We therefore turned to a computer-assisted proof.

\begin{figure}[htbp]
\centering
\includegraphics[width=\textwidth]{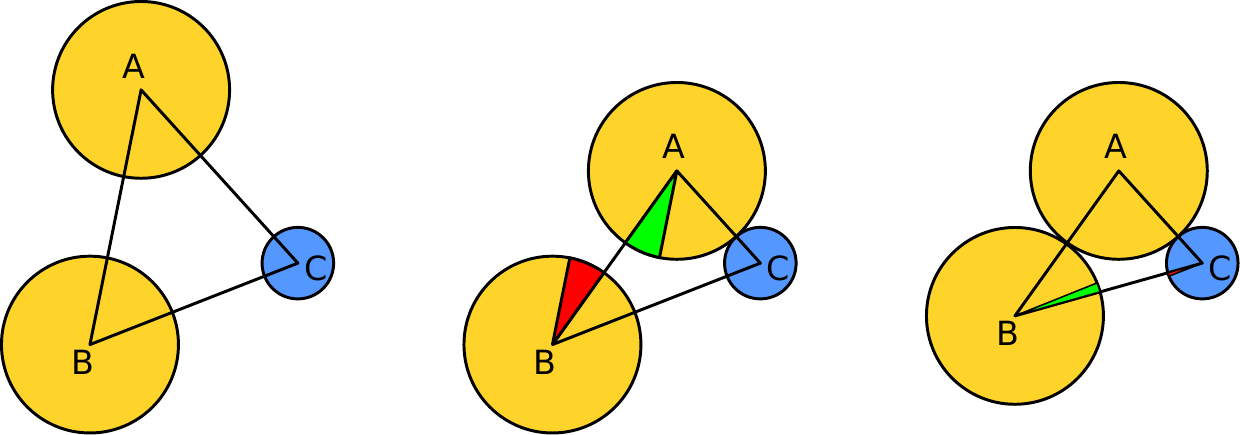}
\caption{
Left: consider a triangle ABC with disks centered on its vertices.
Center: if the disk in A is not smaller than the one in B, then by sliding it towards C as much as possible, the area of ABC decreases.
Further, the slice gained in A (green) is not smaller than the one lost in B (red): both have indeed the same angle \underline{because the sum of the angles of ABC is constant} and the angle in C does not change, and the disk in A has been assumed to be larger or equal to the disk in B.
Right: the same argument allows to slides the disk in B towards the one in A not decreasing the density.
This is the key idea of \cite{FM58} for considering only triangles with two tight edges.
}
\label{fig:intro5}
\end{figure}

Basically, the principle of computer-assisted proof is to make the computer do the verification work that would take too much time for a human -- typically verifying a large number of cases.
A classical example is the proof of the 4-color theorem, generally considered to be the first theorem proved with the help of a computer.
Surprisingly, the verification of a continuum of cases can sometimes also be done with the help of a computer, in particular with the help of {\em interval arithmetic}.
The first such example is the proof by Tucker in 2002 of the 14-th Smale's problem (namely, the Lorenz attractor exhibits the property of a strange attractor) \cite{Tuc02}.
Another remarkable example -- especially for the problem we are interested in here -- is Hales' proof of the aforementioned Kepler conjecture between 1998 and 2005 \cite{Hal05}.
Several other examples of famous problems solved in this way are listed in \cite{Hal14}.
Here, the proof of Theorem~\ref{th:main} consisted in recursively halving the tetrahedron set into millions of small blocks, on each of which the density can be suitably bounded from above.
The interest of the proof lies more in the way in which the computations are carried out and optimized than in its own mathematical content.


\begin{figure}[htbp]
\centering
\includegraphics[width=0.8\textwidth]{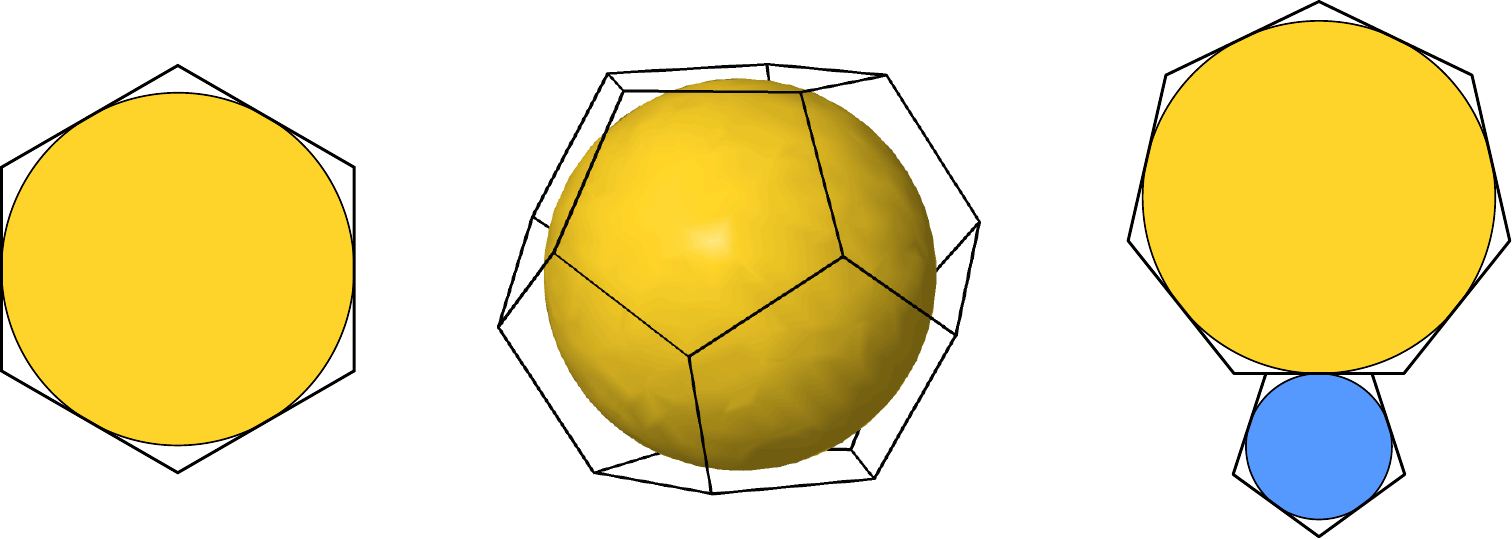}
\caption{
Left: the densest Voronoi cell in a packing of unit disks.
Center: the densest Voronoi cell in a packing of unit spheres.
Right: the density of this figure is larger than the density of any packing of disks of these sizes.
}
\label{fig:intro6}
\end{figure}

To conclude this (long) introduction, let us return briefly to Fejes T\'oth's proof of the maximum density of a packing of unit disks.
This proof is not based on the Delaunay decomposition of plane but on the Voronoi decomposition.
Recall that the {\bf Voronoi decomposition} of a set of points $\mathcal{S}$ is the partition of the plane whose elements, called cells, each consist of a point of $\mathcal{S}$ together with all points of the plane that are closer to that point than to any other point in $\mathcal{S}$.
In \cite{FT43}, Fejes T\'oth proves that the densest cell of a Voronoi decomposition of the centers of a packing of unit disks is a regular hexagon circumscribed to a disk.
The proof follows since the density of such a cell is the optimal density (Fig.~\ref{fig:intro6}, left).

In the same paper, Fejes T\'oth proposed the so-called {\em dodecahedral conjecture}: the densest Voronoi cell in a packing of unit spheres in $\mathbb{R}^3$ is a regular dodecahedron circumscribed to a sphere (Fig.~\ref{fig:intro6}, center).
A computer-assisted proof of this conjecture was established by Hales and McLaughlin in 2010 \cite{HL10}.
Once again, there is a problem of frustration, since the regular dodecahedron, as the regular tetrahedron, does not tile the space.
Nevertheless, the upper bound on the density of a packing of unit spheres derived from the validity of the dodecahedral conjecture was the first progress since \cite{Rog58}, before Hales finally proved Kepler's conjecture shortly afterwards, using a sophisticated hybrid space decomposition between Voronoi cells and Delaunay tetrahedra.

In dimension 2, the use of modified Voronoi diagrams (namely {\em power diagrams}) helped to improve the Florian bound for certain disk sizes.
It was indeed shown in \cite{Bli69} that the density of a packing of disks of sizes in $[r,1]$ was bounded from above by the density in the union of a heptagon circumscribing a disk of size $1$ and a pentagon circumscribing a disk of size $r$ (Fig.~\ref{fig:intro6}, right).
To our knowledge, no similar result has yet been obtained in dimension $3$ or larger.
This might be related to the problem of minimizing the volume of a $n$-faces polyhedra circumscribed to the unit sphere (known as {\em isoperimetric problem for polyhedra} or {\em roundest polyhedra problem}, see, e.g. \cite{Gol35}).
Here, too, the computer is likely to have its say.

The rest of the article is organized as follows.
Section~\ref{sec:settings} formally defines FM-tetrahedra and discusses the main formulas used thereafter.
Section~\ref{sec:strategy} outlines the proof strategy, in particular the use of interval arithmetic.
Section~\ref{sec:epsilon} proves that tetrahedra claimed in Theorem~\ref{th:main} to maximize density among all FM-tetrahedra do so at least locally.
Then, Section~\ref{sec:dim_reduc} explain how to extend this result to the whole space of FM-tetrahedra.
This last point is done by computer, and Section~\ref{sec:computer} gives more details on how.
Last, Appendix~\ref{sec:code} gives an overview of the various programs used in the proof, while Appendix~\ref{sec:proof_FM_faces} provides a technical proof of a proposition appearing in Section~\ref{sec:settings}.


\section{FM-tetrahedra}
\label{sec:settings}

\subsection{Definitions and properties}

A {\bf packing} of spheres in the Euclidean space is an arrangement of interior-disjoint spheres (viewed as solid balls).
In \cite{FM58} a decomposition associated with such a packing was defined as follows:
\begin{itemize}
\item
First, define the {\bf cell} of a sphere as the set of points of the Euclidean space which are closer to this sphere than to any other sphere.
It is star-shaped (with respect to sphere centers) and delimited by hyperboloid sheets (whose foci are the sphere centers).
\item
Then, consider the graph which links the centers of spheres whose cells are adjacent: it yields a decomposition of the Euclidean space in polyhedra.
If such a polyhedron is not a tetrahedron (this is non-generic), decompose it arbitrarily in several tetrahedra.
\end{itemize}
We shall here call {\bf FM-decomposition} of the sphere packing this decomposition.
It is a specific case of {\bf additively weighted Delaunay triangulation} (see e.g. Section 18.3.1 in \cite{BY98}), where the seeds are the sphere centers and weights the sphere radii (the cell decomposition itself is a specific case of {\bf additively-weighted Voronoi diagram}).

The motivation of this paper is to bound from above the density of packings of spheres of specific sizes.
In this context, it is convenient to introduce the notion of {\bf saturated packing}.
These are the packings such that no further sphere can be inserted (in the holes between spheres).
Clearly, it suffices to bound the density of saturated packings in order to bound the density of any other packing (with the same sphere sizes).
The advantage is that the tetrahedra which appear in the FM-decomposition of such a packing have nice properties.
This motivates our first key definition (we shall here use it for spheres of size $1$ and $r=\sqrt{2}-1$ in $\mathbb{R}^3$ but it can be extended in any dimension and for any fixed sizes of spheres):

\begin{definition}
\label{def:fm_tetra}
We call {\bf FM-tetrahedron} any tetrahedron which may appear in the FM-decomposition of a saturated packing of spheres of size $1$ and $r$.
\end{definition}

Implicitly, an FM-tetrahedron has a sphere centered in every vertex (the sphere in the packing it comes from).
These four spheres will be referred to as {\em its} spheres.
Recall from the introduction that the {\bf type} of an FM-tetrahedron is the quadruple formed by the radii of these spheres.
Assuming that these radii are sorted by non-increasing order, this yields $5$ possible types, illustrated in Fig.~\ref{fig:densest_tetrahedra}.

By construction, when several spheres have Voronoi cells which meet in a point, this point is at equal distance from these spheres.
In other words, these spheres are tangent to a new sphere (not in the packing) which is centered at this meeting point and interior disjoint from the spheres in the packing.
This motivates our second key definition:

\begin{definition}
\label{def:support_sphere}
We call {\bf support sphere} of an FM-tetrahedron a sphere which is externally tangent to its four spheres.
\end{definition}

Since, by definition, an FM-tetrahedron can appear in a packing, it always admits a support sphere.
It is usually unique, but there may be some degenerated cases.
For example, if the four spheres are centered on the vertices of a square of size $2$ (the FM-tetrahedron is flat), then any point on the line which goes perpendicularly through the center of this square is the center of a sphere tangent to the four unit spheres.
However, the saturation hypothesis ensures:

\begin{proposition}
\label{prop:bounded_radius}
Support spheres of FM-tetrahedra have radii less than $r$.
\end{proposition}

\begin{proof}
Since a support sphere is interior disjoint from the spheres in the packing, one can insert a further sphere of that size into the packing.
If the radius is $r$ or more it contradicts the saturation hypothesis.
\end{proof}

For the sake of simplicity, we will speak about {\em the} radius of {\em the} support sphere of an FM-tetrahedron, having in mind the radius of the largest possible support sphere in the (rare) cases where more than one such sphere exists.
A straightforward consequence is that the set of FM-tetrahedra, seen as $6$-tuple of their edge lengths, is a bounded subset of $\mathbb{R}^6$:

\begin{proposition}
\label{prop:bounded_edge}
The edge between two vertices $A$ and $B$ of an FM-tetrahedron has length at most $r_A+r_B+2r$, where $r_X$ denotes the radius of the sphere centered in a vertex $X$.
\end{proposition}

\begin{proof}
The tetrahedron has a support sphere of radius $R\leq r$ and center $O$.
The length $AB$ is at most the sum of $AO=r_A+R$ and $OB=R+r_B$ (triangular inequality).
The claim follows.
\end{proof}

Hence, every FM-tetrahedron of type $r_Ar_Br_Cr_D$ corresponds to a point $(ab,ac,ad,bc,bd,cd)$ in the rectangular parallelepiped defined by the product of $[r_X+r_Y,r_X+r_Y+2r]$ for $XY$ in $\{AB,AC,AD,BC,BD,CD\}$.
Conversely, however, not every point in this rectangular parallelepiped corresponds to an FM-tetrahedron.

First, this point has to correspond to a tetrahedron, which is known to be characterized by the two following conditions \cite{DW09}:
\begin{itemize}
\item for $\{x,y,z\}\subset\{a,b,c,d\}$, the triangular inequality $xz\leq xy+yz$ holds, where $xy$ denotes the length of the edge between vertices $X$ and $Y$;
\item the so-called {\em Cayley-Menger} determinant, below, is non-negative:
\[
\left|\begin{array}{cccccc}
0 & ab^2 & ac^2 & ad^2 & 1 \\
ab^2 & 0 & bc^2 & bd^2 & 1 \\
ac^2 & bc^2 & 0 & cd^2 & 1 \\
ad^2 & bd^2 & cd^2 & 0 & 1 \\
1 & 1 & 1 & 1 & 0
\end{array}\right|
\]
\end{itemize}

Second, when the point corresponds to a tetrahedron, there must exist a sphere of radius at most $r$ that is tangent to each sphere of the tetrahedron (the support sphere).
Expressing this second condition is more technical, we will come back to it in detail a bit later (Section~\ref{sec:formulas}).

In \cite{FM58}, an {\bf FM-triangle} is defined as a triangle with three disjoint interior circles centered on its vertices and such that the following properties, proven to be equivalent, hold:
\begin{enumerate}
\item the (fourth) circle which is interior disjoint to the three previous circles, and tangent to them, is smaller than these three circles.
\item no circle intersects the edge connecting the centers of the two other circles.
\end{enumerate}
Above we generalized the first property to define FM-tetrahedra, but we conjecture that the second can also be generalized, namely:

\begin{conjecture}
\label{conj:sphere_sector}
In an FM-tetrahedron, no sphere intersects the triangle connecting the centers of the three other spheres.
\end{conjecture}

Indeed, consider the limit case of a sphere centered in $A$ which is tangent to the three other spheres as well as to the face $BCD$.
Then, the mirror image of the sphere centered in $A$ with respect to the face $BCD$ is tangent to the four spheres of the FM-tetrahedron: it is a support sphere.
Since its radius is $r_A\geq r$, it contradicts Prop.~\ref{prop:bounded_radius}.
Now, intuitively, if the spheres in $B$, $C$ or $D$ are moved apart and the sphere in $A$ moved to enter the face $BCD$, then the radius of the support sphere only increases.
However, this intuition may be hard to formalize (this is already not trivial for FM-triangles in \cite{FM58}).
In this paper, the following weaker property will be sufficient for our purpose (it is used in the proof of Lemma~\ref{lem:sliding}).
A formal proof is provided in Appendix~\ref{sec:proof_FM_faces} (it is already quite complicated).

\begin{proposition}
\label{prop:FM_faces}
Every face of an FM-tetrahedron is an FM-triangle.
\end{proposition}


\subsection{Formulas}
\label{sec:formulas}

We discuss here how to compute the density and the radius of the support sphere of a tetrahedron $T$ given by its edge lengths $(ab,ac,ad,bc,bd,cd)$.

Denote by $\widehat{X}$ the solid angle at a vertex $X$ of a tetrahedron $T$, i.e., the area of the radial projection from $X$ of the face of $T$ that does not contain $X$.
Then the total volume of $T$ covered by its spheres, denoted by $\textrm{cov}(T)$, is at most
\[
\sum_{X\textrm{ vertex of }T} \frac{1}{3}r_X^3\widehat{X}.
\]
This is actually an equality if Conjecture~\ref{conj:sphere_sector} holds.
If we denote by $\textrm{vol}(T)$ the volume of $T$, the {\bf density} of $T$ is
\[
\delta(T):=\frac{\textrm{cov}(T)}{\textrm{vol}(T)}.
\]
The above-mentioned Cayley-Menger determinant is known to be equal to $288\;\textrm{vol}(T)^2$ and allows to compute $\textrm{vol}(T)$.

In order to compute solid angles, we could use Girard's theorem or L'Huillier formula, but Lagrange's formula proved more effective\footnote{The three formulas are referenced on the Wikipedia page \href{https://en.wikipedia.org/wiki/Solid_angle}{solid angle}.}:
\[
\tan\frac{\widehat{A}}{2}=
\frac{6\times \textrm{vol}(T)}{
ab\times ac\times ad+
\overrightarrow{AB}\cdotp\overrightarrow{AC}\times ad+
\overrightarrow{AB}\cdotp\overrightarrow{AD}\times ac+
\overrightarrow{AC}\cdotp\overrightarrow{AD}\times ab
}.
\]
This formula defines $\widehat{A}/2$ only modulo $\pi$, but it is sufficient to find $\widehat{A}$ because we known that $\widehat{A}$ belongs to $[0,2\pi]$.
Moreover, when the right hand side is non-negative, the inequality $\arctan(x)\leq x$ for $x\geq 0$ yields
\[
\frac{\widehat{A}}{2}\leq 
\frac{6\times \textrm{vol}(T)}{
ab\times ac\times ad+
\overrightarrow{AB}\cdotp\overrightarrow{AC}\times ad+
\overrightarrow{AB}\cdotp\overrightarrow{AD}\times ac+
\overrightarrow{AC}\cdotp\overrightarrow{AD}\times ab
}.
\]
That upper bound proved to be very useful for bounding the density when the volume is small.
Indeed, the density is a weighted sum of solid angles divided by the volume.
Using the above formula then simplifies the numerator and denominator by the volume.

The above is the basis for our way of calculating density, which is implemented in the file \verb+routines.cpp+.
The exact formula used is more complex, however, and will be discussed in greater detail in Section \ref{sec:formulas_opt} (in particular, we shall explain the purpose of the file \verb+routines_order1.cpp+).

Let us now consider the computation of the radius of the support sphere.

\begin{proposition}
\label{prop:radius}
The radius $R$ of the support sphere of a tetrahedron is a root of a quadratic polynomial $P=aR^2+bR+c$, where $a$, $b$ and $c$ depend on the edge lengths of this tetrahedron.
\end{proposition}

\begin{proof}
We explicitly compute $P$ in the file \verb+radius.sage+.
The only difficulty is the size of the formula ($P$ turns out to be the sum of $420$ monomials, each of degree $8$), but that is no problem for today's computer algebra software.
Let us here sketch the way we obtain $P$:
\begin{itemize}
\item
We introduce the coordinates of the four vertices A, B, C, D and the center O of the support sphere.
With the radius $R$ of the support sphere this yields $16$ variables.
\item
Up to an isometry, one can assume that $A$ is the origin, $B$ is on the $x$-axis, $C$ is the half-plane $z=0$, $y\geq 0$ and $D$ in the half-space $z\geq 0$.
This fixes $6$ variables, i.e. there are now $10$ variables.
\item
The length of each edge as well as the distance of O to each vertex yield $10$ quadratic equations.
\item
We first determine the coordinates of the four vertices by suitable substitutions.
\item
We then get a system of $3$ linear equations in $R$ and the coordinates of O which allows to express the coordinates of O as functions of $R$.
\item
We substitute these coordinates in one of the equations to get $P$.
\end{itemize}
\end{proof}

\subsection{An optimization problem}

Basically, our problem amounts to finding the maximum of an explicit function (the density) on a compact set (the FM-tetrahedra).
The difficulty comes from two factors:
\begin{itemize}
\item The objective function is complicated: it is a sum of $\arctan$ of rational fractions in $6$ variables.
\item The set over which we want to maximize is not a polytope: it is a subset of a six-dimensional box (Prop.~\ref{prop:bounded_edge}) delimited by three algebraic surfaces: the positivity of the volume (Cayley-Menger determinant), the existence of a support sphere (the discriminant of the polynomial $P$ of Prop.~\ref{prop:radius} must be non-negative) and a support sphere of radius at most $r$ (Prop.~\ref{prop:bounded_radius}).
Figure~\ref{fig:domain} depicts a 3-dimensional cut of such a domain.
\end{itemize}

We conjecture that the density is convex, but this seems hard to check (analytically or by computer).
Moreover, even if it holds it would not simplify the problem that much, because the domain is not a convex polytope.
Indeed, the condition on the support sphere defines a sort of veil whose points seem to be all extremal (see, again, Figure~\ref{fig:domain}).

\begin{figure}[htbp]
\centering
\includegraphics[width=\textwidth]{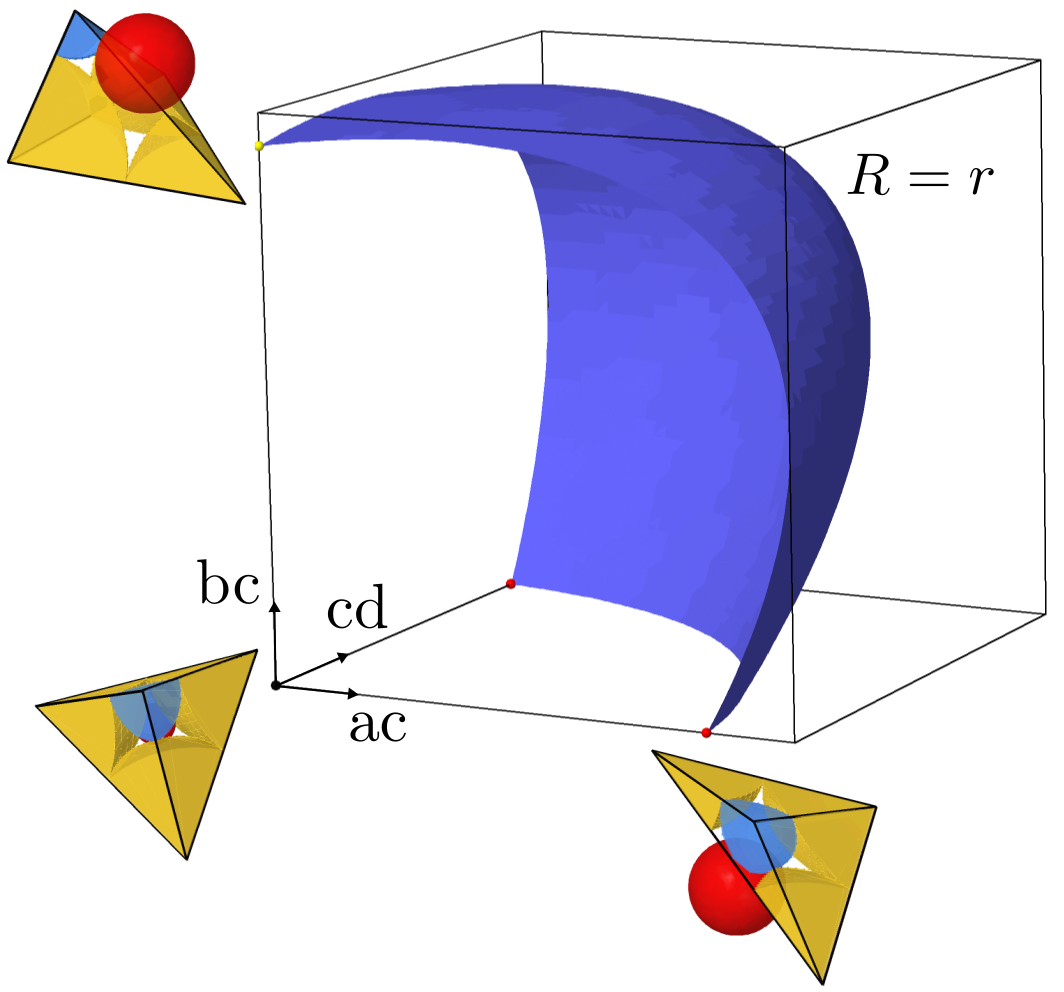}
\caption{
The part of the box below the blue veil is a 3-dim. cut of the set of FM-tetrahedra of type $111r$, namely by the space $(ab,ad,bd)=(2,1+r,1+r)$.
In this case, the positivity of the Cayler-Menger determinant and the existence of a support sphere do not play a role.
The box shows the extreme lengths of $ac$, $bc$ and $cd$ (Prop.~\ref{prop:bounded_edge}) and the blue veil shows the tetrahedra whose support sphere has radius $r$ (Prop.~\ref{prop:bounded_radius}).
Some extreme tetrahedra are depicted: the tight one (black point, bottom left), the ones stretched along an 11-edge (the two red points -- only one tetrahedron is depicted since the other is similar) and the one stretched along the 1r-edge $BC$ (the yellow point). 
}
\label{fig:domain}
\end{figure}

In what follows, we will often speak about tight or stretched FM-tetrahedra.
An FM-tetrahedron is said to be {\bf tight} if its four spheres are pairwise adjacent.
There is thus only one tight tetrahedron of each type.
An FM-tetrahedron is said to be {\bf stretched} if its spheres are pairwise adjacent except along ``some'' edges that have been stretched so that its support sphere reached radius $r$.
Formally, it is thus not an FM-tetrahedron (the radius should be less than $r$) but a limit of FM-tetrahedra.
The meaning of ``some'' edges is informal; most of the time it means only one, and usually we specify which edges have been stretched.
The idea is that these are extreme FM-tetrahedra of the domain, hence good candidates to maximize the density.
Among the densest tetrahedra depicted in Fig.~\ref{fig:densest_tetrahedra}, the upper three ones are tight, the bottom-left one is stretched along an 11-edge while the bottom-right one has two stretched rr-edges (their lengths are characterized by the fact that they are equal).

\subsection{Towards a 3D Florian's bound?}
 
In order to extend the Florian bound discussed in the introduction to the 3-dimensional case, we would like to determine which FM-tetrahedron is the densest, that is, which one plays the role of the triangle with two small and one large disks that are all pairwise adjacent (Fig.~\ref{fig:intro3}).
Natural candidates are the tight FM-tetrahedra and, maybe, some stretched ones.
Unfortunately, the densest of these tetrahedra varies for different values of $r$, as illustrated in Figure~\ref{fig:florian3D_curves}.

\begin{figure}[htbp]
\includegraphics[width=\textwidth]{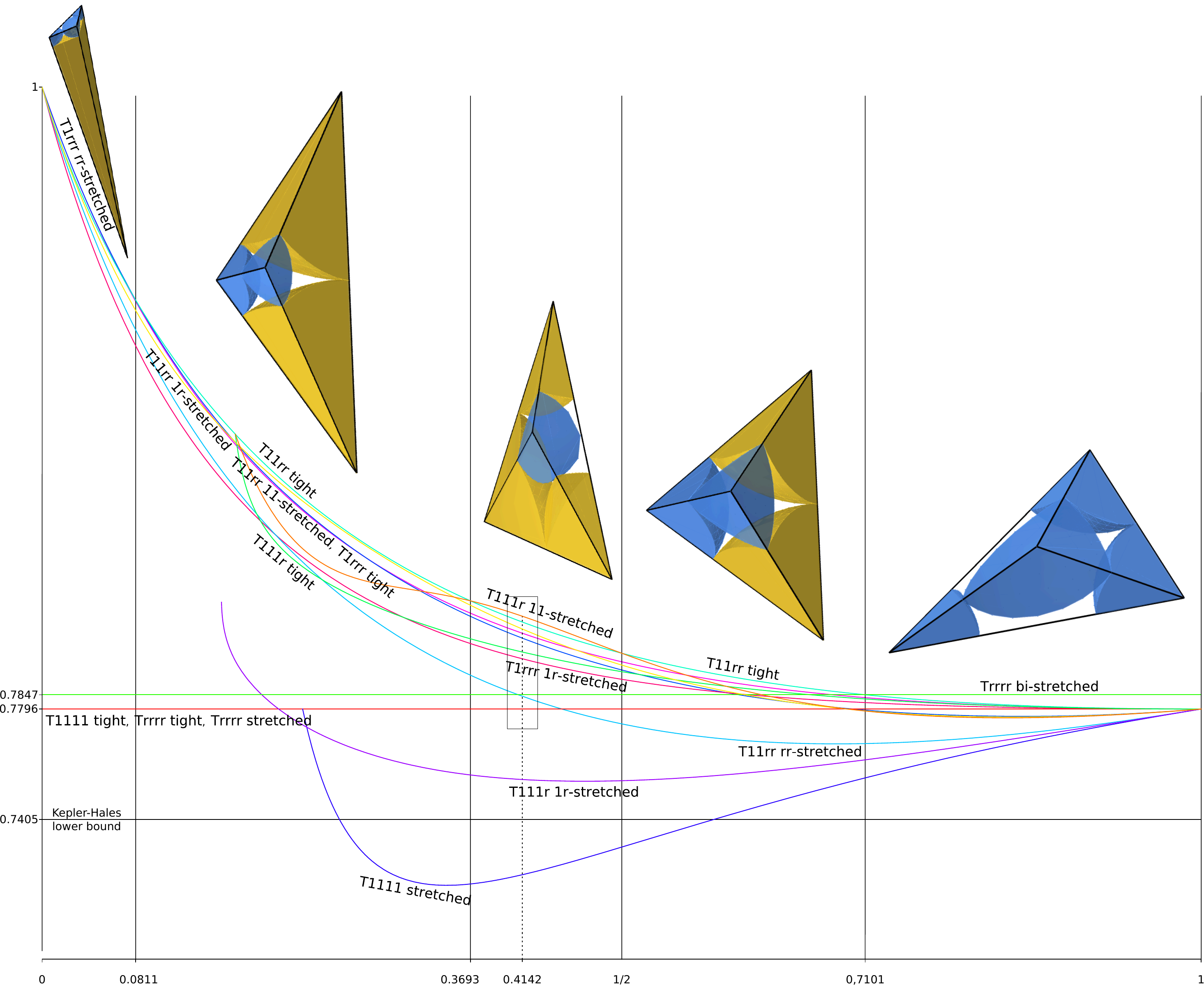}
\caption{
Density of some extremal FM-tetrahedra as a function of $r$.
Which is the densest depends on $r$.
The black frame around $r=0.4142$ is magnified in Fig.~\ref{fig:florian3D_curves_zoom}.
For large $r$, the densest tetrahedron has only $4$ contacts, while two of its edges are equally stretched, so that its support sphere has size $r$ (hence it is also on the boundary of the domain of FM-tetrahedra).
}
\label{fig:florian3D_curves}
\end{figure}

We conjecture that each of the five tetrahedra depicted in Figure~\ref{fig:florian3D_curves} is, for any radius $r$ between the two vertical lines flanking this tetrahedron, the densest among all FM-tetrahedra with spheres of radius $1$ and $r$.
In this paper, we focus on the $r=\sqrt{2}-1$ case, see zoom in Figure~\ref{fig:florian3D_curves_zoom}.

\begin{figure}[htbp]
\centering
\includegraphics[width=\textwidth]{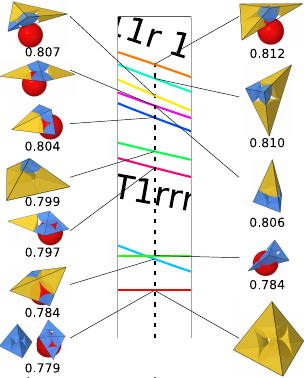}
\caption{
Zoom on the black frame in Fig.~\ref{fig:florian3D_curves}.
The extreme tetrahedra for $r=\sqrt{2}-1$ (vertical dotted line) corresponding to each curve are depicted.
For stretched ones, the support sphere (which has radius $r$) is drawn in red.
Their (rounded) density is reported.
Theorem~\ref{th:main} states that each of the five tetrahedra on the right is the densest among all FM-tetrahedra of the same type.
}
\label{fig:florian3D_curves_zoom}
\end{figure}

\section{Proof strategy}
\label{sec:strategy}

\subsection{Interval arithmetic}

Real numbers are usually represented on a computer as so-called {\em floats}, which are floating point numbers with a finite precision.
Of course floats cannot represent all real numbers, so there are rounding issues.
Floating point arithmetic is useful to get good approximations, but does not offer rigorous results.

The principle of interval arithmetic is to represent a real number by an interval that contains it and whose endpoints are {\em exactly representable}, that is, floats handled by the computer.
Then, this property is maintained during arithmetic calculations, that is, if $x_1,\ldots,x_k$ are intervals and $f:\mathbb{R}^k\to\mathbb{R}$ is a $k$-ary function, then computing $f(x_1,\ldots,x_k)$ yields an interval which contains at least the set
\[
\{f(y_1,\ldots,y_k)~|~y_1\in x_1,\ldots,y_k\in x_k\}.
\]
Because the function $f$ can be very complicated and the above set may not be an interval whose endpoints are representable on the computer, the computed interval is usually not optimal.
Its diameter depends on the implementation of the used interval arithmetic library.
However, it contains every valid result and this is what allows to perform certified computations to prove theorems.
We illustrate this with SageMath \cite{sage}, a free computer algebra software whose syntax should be clear enough:
\begin{verbatim}
sage: pi=RealDoubleField()(4*arctan(1))
sage: pi
3.141592653589793
sage: sin(pi)
1.2246467991473515e-16
\end{verbatim}
From a mathematical viewpoint, $4\;\arctan{1}$ is equal to the transcendental number $\pi$, whose sine is equal to zero.
However, in floating point arithmetic, $4\;\arctan{1}$ yields an approximation of $\pi$ which is printed \verb+3.141592653589793+ (actually the exact float is slightly different\footnote{Namely, in that case: $3.141592653589793115997963468544185161590576171875$.}, but SageMath prints only an approximation of this approximation).
Then, if we compute the sine of this float, we get another float, that is thus an approximation of the sine of the approximation of $\pi$.
It is rather close to zero but not equal to zero!
Compare with the following lines:

\begin{verbatim}
sage: pi=4*arctan(RealIntervalField()(1))
sage: pi
3.141592653589794?
sage: pi.endpoints()
(3.14159265358979, 3.14159265358980)
sage: sin(pi).endpoints()
(-3.21624529935328e-16, 1.22464679914736e-16)
\end{verbatim}

Here, $4\;\arctan{1}$ is computed using interval arithmetic (e.g. with a Taylor expansion of arctangent).
So we get an interval which contains $\pi$.
SageMath prints an approximation, with a question mark to indicate that the last printed digit is questionable.
One can ask to print the endpoints.
Then, the sine of this interval is an interval which indeed contains $0$.

\subsection{Checking a strict inequality}
\label{sec:interval_dichotomy}

Usually, we try to keep the intervals as small as possible to get good bounds of the reals we are interested in.
But large intervals can also be interesting when we want to check inequalities over compact sets.
It is a different use of interval arithmetic, with almost the opposite philosophy.

Assume we are given a function $f:\mathbb{R}\to\mathbb{R}$ and a closed real interval $X$ and we want to check whether $f$ is positive on $X$.
If computing $f(X)$ in interval arithmetic yields an interval whose lower bound is positive, then the answer is yes.
On the contrary, if the upper bound of this interval is negative, then the answer is no.
However, if the interval contains $0$, then we cannot conclude.
In this case, we can halve $X$ and check recursively the positivity of $f$ over each half interval.
Indeed, although this is most often implicit, a good interval arithmetic library will return for $f(X)$ an interval whose diameter tends towards $0$ when the diameter of $X$ tends towards $0$.
Hence, provided that $f$ has no zero over $X$, the process will eventually end.
Illustrated with SageMath (\verb+X.bisection()+ splits the interval \verb+X+ in two):

\begin{verbatim}
def is_f_positive_over_X(f,X):
    if f(X).lower()>0:
        return True
    elif f(X).upper()<=0:
        return False
    else:
        (X1,X2)=X.bisection()
        return self(f,X1) and self(f,X2)
\end{verbatim}

This extends without difficulty to the case of a function $f$ of several variables.
However, note that adding one variable doubles the number of recursive calls, hence computational time issues may appear.

\subsection{Optimal points}

The above approach allows us to prove a strict inequality $f>0$ over a set $X$.
But if $f$ admits a zero $x_0\in X$, then $x_0$ will stay indefinitely in the interior of an interval (no matter how much it has been subdivided) and applying $f$ to this interval will yield an interval with $0$ in its interior (this happens e.g. if $x_0$ is not exactly representable on the computer).
In such a case, the above program will run forever in ever-smaller intervals.

To get around this problem, just go one step further.
Namely, assuming that $f$ is differentiable and $X$ is convex, the mean value theorem states that for any $x\in X$, there exists $y$ on the segment from $x_0$ to $x$ such that
\[
f(x)=f(x_0)+(x-x_0)f'(y).
\]
Hence, if we can find $\varepsilon>0$ such that the interval $f'([x_0,x_0+\varepsilon])$ is positive, then this ensures $f(x)\geq f(x_0)$ on $[x_0,x_0+\varepsilon]$.
If we can also proceed similarly on $[x_0-\varepsilon,x_0]$, then we get an explicit neighborhood of $x_0$ on which $f$ is non-negative.
The above recursive checking is then modified as follows: whenever we get a set included in $[x_0-\varepsilon,x_0+\varepsilon]$, we answer that $f$ is non-negative and stop the recursion.
Note that it is essential to have calculated an explicit neighborhood and not just to have shown that $f$ admits a local minimum at $x_0$.

Again, this extends without difficulty to the case of a function $f$ of several variables.
In order to prove $f(x)\geq f(x_0)$ over $B=x_0+[0,\varepsilon]^d$,
we show that $\nabla f(B)$ is positive (that is, every entry is a positive interval).
Indeed, for $x\in B$, $x\neq x_0$, the mean value theorem for $g:t\mapsto f((1-t)x_0+tx)$ ensures the existence of $t\in(0,1)$ s.t.
\[
f(x)-f(x_0)=g(1)-g(0)=g'(t)=\nabla f((1-t)x_0+tx)\cdotp (x-x_0).
\]
Since $\nabla f((1-t)x_0+tx)$ is positive because $(1-t)x_0+tx\in B$ and $x-x_0$ is non-negative and non-zero because of the shape of $B$, this yields $f(x)>f(x_0)$.


\section{Near optimal tetrahedra}
\label{sec:epsilon}

Let us denote by $T^*_t$ the tetrahedron of type $t$ claimed in Theorem~\ref{th:main} to maximize the density.
This is thus a point in $\mathbb{R}^6$ (the length of the six edges).
This section provides, for each type $t$, an explicit neighborhood of $T_t^*$ over which $T_t^*$ indeed maximizes the density.
All the computations are detailed in the file \verb+local.sage+.

\subsection{Tight tetrahedra}
\label{sec:tights}

For the tight tetrahedra (types $1111$, $11rr$ and $1rrr$), we look for a neighborhood\footnote{Strictly speaking, this is the intersection of a neighborhood of the tight tetrahedron with the definition domain of the tetrahedra. For the sake of simplicity, we will nevertheless speak about a neighborhood.} $B_\epsilon:=[0,\varepsilon]^6$.
Indeed, the edges cannot be shortened because spheres cannot intersect.
Since all the partial derivatives of the density turn out to be negative for each of these tight tetrahedra, such a neighborhood exists.
To find a suitable $\varepsilon$, we use interval arithmetic to compute the vector $\nabla\delta(T_t^*+B_\varepsilon)$ for $t\in\{1111,11rr,1rrr\}$ and check that it is negative.
The mean value theorem then ensures that the density over $T_t^*+B_\varepsilon$ is maximal in $T_t^*$.
For the sake of simplicity, here and in the following we will always look for the largest $\varepsilon$ of the form $1/k$, $k\in\mathbb{N}$.
The values for the tight types are reported in the first three columns of Table~\ref{tab:epsilon}.

\begin{table}[htbp]
\centering
\begin{tabular}{|c|c|c|c|c|c|}
\hline
type & $1111$ & $11rr$ & $1rrr$ & $111r$ & $rrrr$\\
\hline
$\varepsilon$ & $1/46$ & $1/203$ & $1/148$ & $1/686$ & $1/1693$\\
\hline
\end{tabular}
\caption{
Values of $\varepsilon$, for each type $t$, such that no FM-tetrahedron of type $t$ whose edge lengths are within $\varepsilon$ from those of $T_t^*$ is denser than $T_t^*$.
}
\label{tab:epsilon}
\end{table}

\begin{figure}[htbp]
\includegraphics[width=\textwidth]{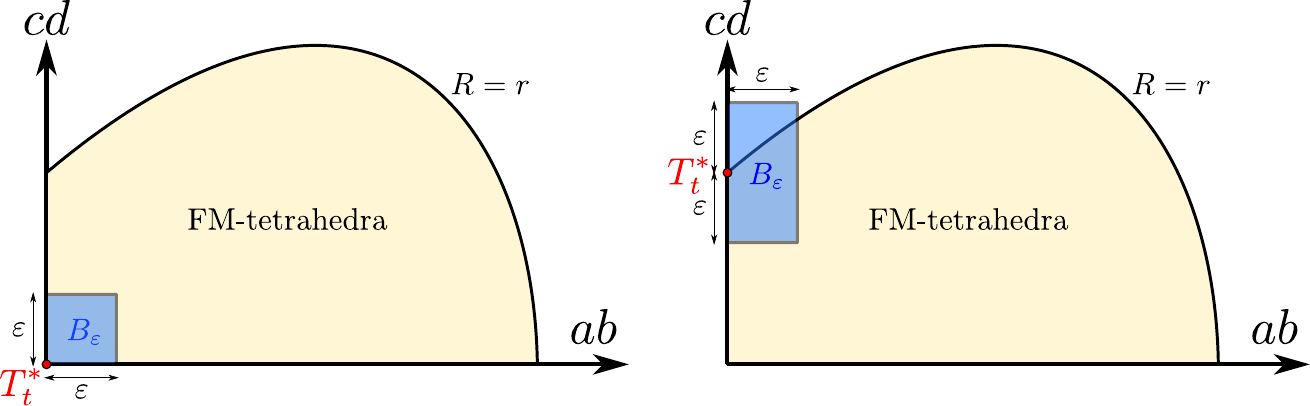}
\caption{
The set of FM-tetrahedra is delimited by the surface $R=r$, where $R$ is the radius of the support sphere, and the six hyperplanes ``the edge length between two spheres is at least the sum of the radii of these spheres'' ($ab$ denotes the length of the edge $AB$).
This set is depicted here, using only two variables for the sake of clarity (compare also with Fig.~\ref{fig:domain}).
The three tight tetrahedra are at the intersection of the six hyperplanes (left).
The two stretched ones are at the intersection of the surface $R=r$ with $4$ (type $rrrr$) or $5$ (type $111r$) hyperplanes (right).
In this latter case, there are tetrahedra in $T_t^*+B_\varepsilon$ which are denser than $T_t^*$, but they are not FM-tetrahedra!
}
\label{fig:local_maxima}
\end{figure}

\subsection{The stretched tetrahedron $T_{111r}^*$}
\label{sec:T111r}

For the stretched tetrahedron of type $111r$, $T_{111r}^*$, the stretched edge - say cd - can be either lengthened or shortened (Fig.~\ref{fig:local_maxima}, right).
To prove that the density has a local maximum in $T_{111r}^*$, we thus look for $\varepsilon>0$ such that
\begin{itemize}
\item on $B^+_\epsilon:=[0,\varepsilon]^5\times[0,\varepsilon]$, all the partial derivatives of the density are negative;
\item on $B^-_\epsilon:=[0,\varepsilon]^5\times[-\varepsilon,0]$, the first five partial derivatives of the density are negative while the last one, along cd, is positive.
\end{itemize}
The same approach as for tight tetrahedra shows that $T_{111r}^*$ maximizes the density over $B^-_\epsilon$ for $\varepsilon=1/2225$.

Sadly, this fails over $B^+_\epsilon:=[0,\varepsilon]^5\times[0,\varepsilon]$, which turns out to contain tetrahedra which are denser than $T_{111r}^*$.
However, these tetrahedra may be not FM-tetrahedra, because $B_\epsilon$ intersects the surface $R=r$ that bounds the set of FM-tetrahedra we are interested in.
Since it is hard to study the density over the intersection of $B^+_\epsilon$ and $R\leq r$, we shall use the method of {\em Lagrange multipliers}.
Namely, we introduce the function
\[
f:=\delta-\alpha R,
\]
where $\alpha>0$ is the Lagrange multiplier (that we shall choose carefully later).
Then, for any FM-tetrahedron $T$:
\[
f(T_{111r}^*)\geq f(T)
\Leftrightarrow
\delta(T_{111r}^*)-\delta(T)\geq \alpha ( \underbrace{R(T_{111r}^*)}_{=r}-\underbrace{R(T)}_{\leq r})\geq 0.
\]
We therefore reduce to showing that $f$ is maximal at $T_{111r}^*$ over $B^+_\epsilon$.
The benefit is that the Lagrange multiplier will ``penalize'' tetrahedra in $B_\varepsilon$ with $R(T)>r$, especially those denser than $T_{111r}^*$ which were problematic above.

Computations shows that for $\alpha=1.1$ and $\varepsilon=1/9940$, the partial derivatives of $f$ are all non-negative on $B^+_\varepsilon$.
This proves that $f$ is maximal at $T_{111r}^*$ on $B_\varepsilon$, hence $\delta$ is maximal at $T_{111r}^*$ on the FM-tetrahedra of $B^+_\varepsilon$.

\subsection{The stretched tetrahedron $T_{111r}^*$ again}
\label{sec:T111r_order2}

The value $\varepsilon=1/9940$ obtained in the previous section is quite small.
Let us here show how we can improve it.
This is actually not really necessary for $T_{111r}$, but this will be for $T_{rrrr}$ in the next section.
The idea is simply to use Taylor-Lagrange one order further, that is, at order two in $T^*_{111r}$, seen as a point in $\mathbb{R}^6$:
\[
f(T)=f(T_{111r}^*)+\Delta f(T_{111r}^*)\cdotp h + \tfrac{1}{2}h\cdotp H(T') \cdotp h^\mathrm{T},
\]
where $h:=T-T^*_{111r}$, $T'$ is a point on the segment joining the two points $T$ and $T^*_{111r}$ and $H$ is the Hessian of $f$, that is, the square matrix whose entries are the second-order partial derivatives of $f$.

By replacing $h$ by $[0,\varepsilon]^6$. and $T'$ by $B^+_\varepsilon$ in the above formula, we get an interval which contains $f(T)-f(T_{111r}^*)$ for any $T\in B^+_\varepsilon$.
Unfortunately, the second order terms have coefficients that may be positive and thus yield an interval whose upper bound is positive.
Hence, even if $\Delta f(T_{111r}^*)\cdotp h$ yields an interval whose upper bound is $0$, the sum over first and second order terms yields an interval which contains zero (in its interior) and we cannot conclude.
We need to be slightly more careful.

{\bf Fix an arbitrary value $\varepsilon>0$} (we shall explain how at the end).
Replace $H(T')$ by $H(B^+_\varepsilon)$ in the above formula (as in the previous attempt).
But now, write $h=(h_1,\ldots,h_6)$ and compute $f(T)-f(T_{111r}^*)$ as a function of the $h_i$'s.
This yields a polynomial $z$ with monomials in $h_i$ which come from the first order and have non-positive coefficients, and monomials in $h_ih_j$ which come from the second order and may have arbitrary coefficients.
We will group a term in $h_ih_j$ with the term in $h_i$ (or $h_j$), so that the non-positive coefficient of $h_i$ overcomes, for $h_j$ small enough, the coefficient of $h_ih_j$.
Let us explain this on a toy example with only two variables:
\[
z:=2h_1^2+3h_1h_2+4h_2^2-h_1-2h_2.
\]
A short computation shows that replacing $h_1$ and $h_2$ by $[0,\varepsilon]$ in the above expression yields $z=[-3\varepsilon, 2\varepsilon+7\varepsilon^2]$, which contains $0$.
Instead, rewrite
\[
z=h_1(2h_1+3h_2-1)+h_2(4h_2-2).
\]
Now, replacing $h_1$ and $h_2$ by $[0,\varepsilon]$ in the coefficients $(2h_1+3h_2-1)$ of $h_1$ and $(4h_2-2)$ of $h_2$ respectively yields $[-1,7\varepsilon-1]$ and $[-2,4\varepsilon-2]$, which are both non-positive for $\varepsilon$ as large as $1/7$.
Since $h_1$ and $h_2$ are non-negative, this ensures that $z$ is non-positive, as wanted.

Rewriting $z$ as above, we ``give'' the coefficient $3$ of the term $h_1h_2$ to $h_1$.
But we could also give it to $h_2$ or share it between $h_1$ and $h_2$.
If we share the coefficient of $h_ih_j$ among $h_i$ and $h_j$ proportionally to the absolute value of the coefficients of $h_i$ and $h_j$, then we get
\[
z=h_1(2h_1+h_2-1)+h_2(2h_1+4h_2-2).
\]
The previous argument ensures that $z$ is non-positive for $\varepsilon$ as large as $1/3$.
This is the heuristic we use in the file \verb+local.sage+ (function \verb+negativity+).

The only problem which may appear is that, for some $i$ and $j$ (possibly $i=j$), there is a term in $h_ih_j$ with a positive coefficient but no term in $h_i$ or $h_j$ with a negative coefficient to counterbalance.
This does not happen here because $\Delta f(T_{111r}^*)$ has only negative entries, but we shall be careful about that for $T_{rrrr}^*$ in the next section.

What about the {\bf arbitrary value $\varepsilon>0$} fixed initially?
It was used to compute $H(B^+_\varepsilon)$, hence the coefficients of the polynomial $z$ in the $h_ih_j$'s and $h_i$'s.
These coefficients, in turn, yield a constraint on $\varepsilon$ in order to have $z$ non-positive (e.g. at most $1/3$ in the toy example when coefficients are proportionally shared).
The smaller the initial value of $\varepsilon$ is, the weak is the constraint we get on $\varepsilon$ for the non-positivity of $z$.
The aim is to find the equilibrium point — namely, to choose $\varepsilon$ as large as possible initially, while ensuring that it is not subsequently reduced by the resulting final constraint.
We can e.g. proceed by a trial-and-error approach, replacing each time the initial value of $\varepsilon$ by the average of the previous initial value and the previous subsequently value. 
In the case of $T_{111r}^*$, this showed that fixing initially $\varepsilon=1/686$ eventually yields the constraint $\varepsilon\leq 1/686$.

Now, the limiting value is $\varepsilon=1/2225$, obtained for $B_\varepsilon^-$ in the previous section!
But we can also use the approach described above.
This eventually yields $\varepsilon=1/417$, beyond the previous value $1/686$.
This latter is thus the limiting value; it is reported in Table~\ref{tab:epsilon}.

\subsection{The stretched tetrahedron $T_{rrrr}^*$}
\label{sec:Trrrr}

The stretched tetrahedron $T_{rrrr}^*$ has two stretched (incident) edges, say bd and cd, that can be both contracted or elongated.
We thus search for $\varepsilon>0$ such that:
\begin{itemize}
\item
on $B^{++}_\varepsilon:=[0,\varepsilon]^4\times [0,\varepsilon]\times [0,\varepsilon]$, all the partial derivatives of the density are negative;
\item
on $B^{--}_\varepsilon:=[0,\varepsilon]^4\times [-\varepsilon,0]\times [-\varepsilon,0]$, the first four partial derivatives of the density are negative while those along bd and cd are positive;
\item
on $B^{+-}_\varepsilon:=[0,\varepsilon]^4\times [0,\varepsilon]\times [-\varepsilon,0]$, the first five partial derivatives of the density are negative while the one along cd is positive;
\item
on $B^{-+}_\varepsilon:=[0,\varepsilon]^4\times [-\varepsilon,0]\times [0,\varepsilon]$, all the partial derivatives of the density are negative but the one along bd that is positive.
\end{itemize}

For $B^{++}_\varepsilon$ and $B^{--}_\varepsilon$, the same approach as for $T_{111r}^*$ at order $1$ with a well-chosen Lagrange multiplier (section~\ref{sec:T111r}) works.
Namely, we get $\varepsilon=1/1133$ with Lagrange multipliers $\alpha=1.45$ for $B^{++}_\varepsilon$,
and $\varepsilon=1/1519$ with Lagrange multipliers $\alpha=0.76$ for $B^{--}_\varepsilon$.
We could get larger $\varepsilon$ by going to order $2$ as in section~\ref{sec:T111r_order2}, but it is not worth: the limiting case will be $B^{+-}_\varepsilon$ and $B^{-+}_\varepsilon$.

The cases of $B^{+-}_\varepsilon$ and $B^{-+}_\varepsilon$ are symmetric.
Let us focus e.g. on $B^{+-}_\varepsilon$.
The main problem is that, because of this symmetry, the last two partial derivatives of $f:=\delta-\alpha R$ (w.r.t. bd and cd) are equal, no matter which Lagrangian multiplier $\alpha$ we take.
We thus cannot have one positive and the other negative.
However, the sign of each derivative can be achieved as wanted independently with two different Lagrangian multipliers.
Therefore, we split $B^{+-}_\varepsilon$ in two sets:
\begin{itemize}
\item The cone $C_\varepsilon^<$ between $bd=0$ and the diagonal $bd+cd=0$,
\item The cone $C_\varepsilon^>$ between the diagonal $bd+cd=0$ and $cd=0$.
\end{itemize}
On each cone, we follows the approach described in Section~\ref{sec:T111r_order2}.
We found (the computations are detailed in the file \verb+local.sage+):
\begin{itemize}
\item
For $\alpha=0.954$, $f=\delta-\alpha R$ is maximal in $T_{rrrr}^*$ over $C_\varepsilon^<$ for $\varepsilon=1/1693$;
\item
For $\alpha=0.958$, $f=\delta-\alpha R$ is maximal in $T_{rrrr}^*$ over $C_\varepsilon^>$ for $\varepsilon=1/1691$;
\end{itemize}
The limiting value (including the values over $B_\varepsilon^{++}$ and $B_\varepsilon^{--}$) is $\varepsilon=1/1693$, reported in Table~\ref{tab:epsilon}.

The two values for $\alpha$ provide a fairly tight bracket around the value $0.9559378928267108\ldots$ for which the partial derivatives of $f$ w.r.t. bd and cd vanish, and the decay of $f$ along the diagonal is the steepest.
The idea is that by tilting $\alpha$ slightly to one side or the other of this pivot value, one can obtain the desired sign for the partial derivative w.r.t. bc or cd without significantly harming the decay of $f$ along the diagonal.
All these values (the $\alpha$'s and $\varepsilon$'s) have been obtained through trial and error.

\section{Away from optimal tetrahedra}
\label{sec:dim_reduc}

\subsection{Dimension reduction}

Now we have to check over the whole 6-dimensional space of FM-tetrahedra, beyond the neighborhood of optimal tetrahedra considered in the previous section, whether the density is indeed bounded as claimed in Theorem~\ref{th:main}.
As explained in Section~\ref{sec:strategy}, we want to partition this space into sufficiently small blocks of tetrahedra over which the density, computed with interval arithmetic, is provably less than the claimed bound (that is, the upper endpoint of the computed interval is less than the lower endpoint of the claimed bound).

Due to the complexity of the involved formulas (density and radius of the support sphere) we have to divide the space into many small blocks.
Since in dimension $d$, each block is divided into $2^d$ blocks, the dimension of the space appears to be challenging.
We managed to do it for the tetrahedra of types 1111 and 111r but we gave up for the other ones.
For type 111r, it amounts to check almost $5$ billions of blocks and took us more than $11$ days of CPU time (11th Gen Intel Core i5-1145G7).
In retrospect, we estimated on the basis of the results presented in Section~\ref{sec:final} that the remaining cases would have taken at least 10 times longer.

The idea of dimension reduction is to find a transformation which moves every point of the space to a point in a subspace of lower dimension so that the function to maximize (here, the density) does not decrease.
It then suffices to bound from above the function on the lower dimensional space to get a bound on the whole space.

Probably the first idea which comes is to shrink uniformly the edges of the tetrahedra (homothety) until two spheres (whose size is not modified) come into contact.
Since the volume of the tetrahedron decreases and the volume covered by the spheres does not change, the density increases.
When a contact between two spheres is reached, the length of the edge between these two spheres is fixed: we are in a 5-dimensional space.
This idea is used in \cite{FM58} in the case of triangles (to reduce from dimension $3$ to $2$).
Unfortunately, for some tetrahedra (Fig.~\ref{fig:homothety}), this transformation may exit the space of FM-tetrahedra because the radius of the support sphere either increases beyond $r$ or even ceases to exist.
Such a transformation is thus not suitable here.

\begin{figure}[htbp]
\centering
\includegraphics[width=0.8\textwidth]{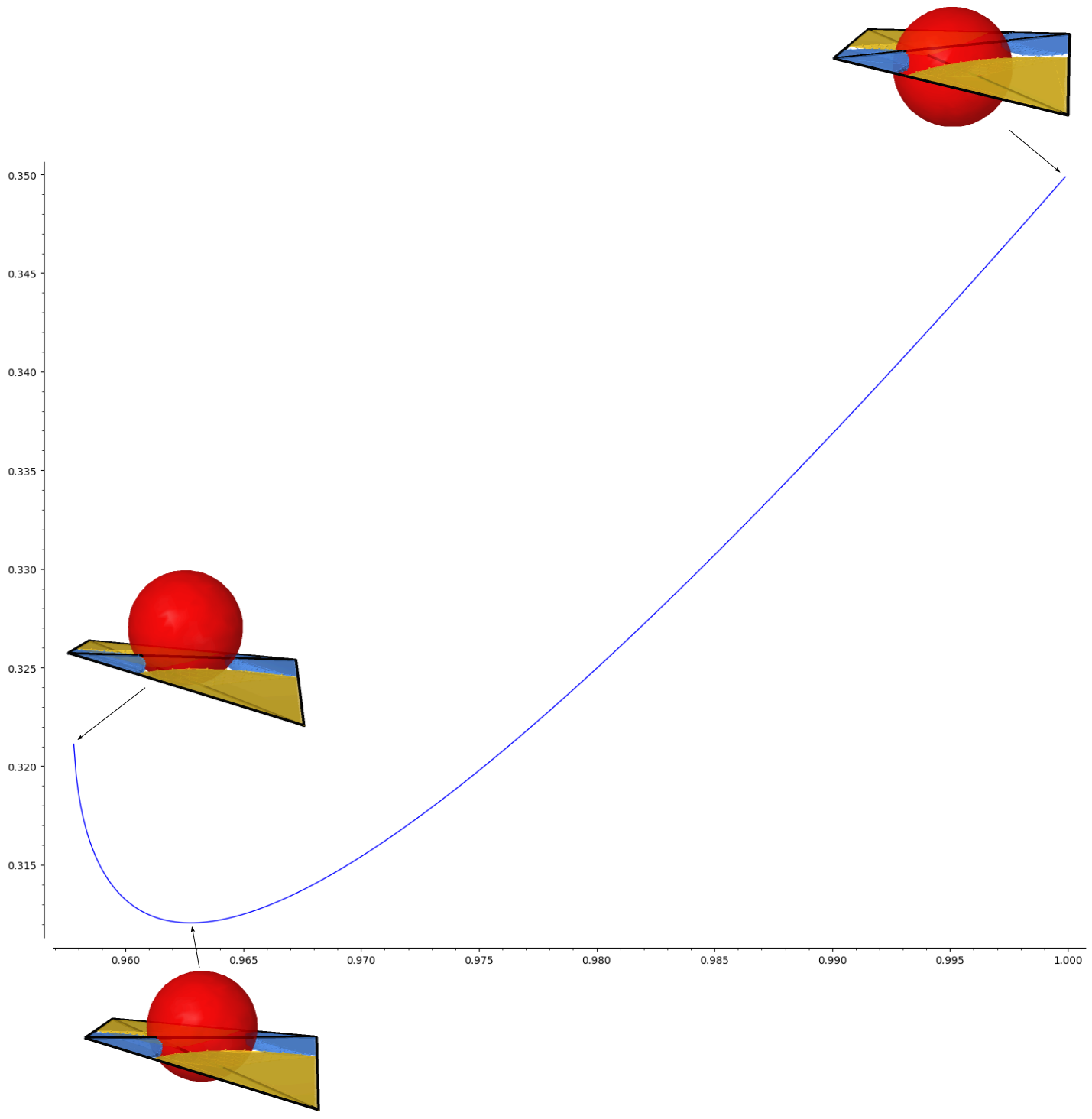}
\caption{
Consider the tetrahedron of type 11rr with edge lengths $(2.7,1.6,1.6,1.5,1.5,1.5)$ (rightmost tetrahedron).
It has a support sphere (red sphere) of radius $R\approx 0.35$.
It is thus an FM-tetrahedron (even if it is quite flat -- such tetrahedra are usually called ``sliver'').
When shrinking the edges of this tetrahedron by a factor $x<1$, the radius decreases until $R\approx 0.312$ for $x\approx 0.963$ (bottom most tetrahedra).
But it then increases until $R\approx 0.323$ for $x\approx 0.958$ (leftmost tetrahedra) and is not anymore defined for smaller $x$: we left the space of FM-tetrahedra before any contact between spheres have been reached!
}
\label{fig:homothety}
\end{figure}

It is not very hard to imagine transformations which increase the number of contacts between spheres (i.e. reduce the dimension) and seem to decrease the density.
For example, one can contract one edge until the spheres at its endpoints get into contact.
Or we can slide a sphere along an edge until it gets into contact with the sphere centered in the other endpoint.
In order to be sure to not increase the radius of the support sphere, we could also ``roll'' a sphere on the support sphere towards another sphere.
However, it may be hard to prove -- if ever true -- that these transformations indeed increase the density and do not exit the space of FM-tetrahedra.

In what follows, we shall prove that sliding spheres towards the support sphere allows us to consider a $5$-dim. space.
Unfortunately, despite numerous attempts, we were not able to reduce the dimension further.
Mention that sliding was already used for triangles in \cite{FM58}, as well as for tetrahedra when $r=1$ in \cite{Hal97}.
One of the key points of this transformation is that its action on the solid angles at the vertices of a tetrahedron or triangle is relatively simple (especially in a triangle, since the sum of its angles is constant, which is unfortunately not true for a tetrahedron).

Actually, it will be more convenient to work with an additive analogue of the density, namely the compression.
Recall that $\textrm{vol}(T)$ denotes the volume of the tetrahedron $T$ and $\textrm{cov}(T)$ denotes the volume of $T$ covered by its spheres.
Then:

\begin{definition}
For $\delta>0$, the {\bf $\delta$-compression} of a tetrahedron $T$ is defined by
\[
\Gamma_\delta(T):=\textrm{cov}(T)-\delta\textrm{vol}(T).
\]
\end{definition}

\subsection{Sliding and the compression}

The following lemma shows that if a somewhat technical inequality is verified, then compression behaves well with sliding:

\begin{lemma}
\label{lem:sliding}
Let $ABCD$ be a tetrahedron (not necessarily FM) and $\delta>0$.
Assume that, for $d$ defined by $\delta d^3=1$, the following inequality holds
\[
\frac{(r_B+r_C)(r_A+(1-d)\max(r_B,r_C))^2}{12(1+r)}
>
\frac{d\min(r_B,r_C)(d-1)^2(r_B+r_C)^2}{2
\min(r_A,r_B,r_C)}.
\]
Then, sliding $A$ towards $D$ does not decrease the $\delta$-compression of $ABCD$.
\end{lemma}

\begin{figure}[htbp]
\centering
\includegraphics[width=0.9\textwidth]{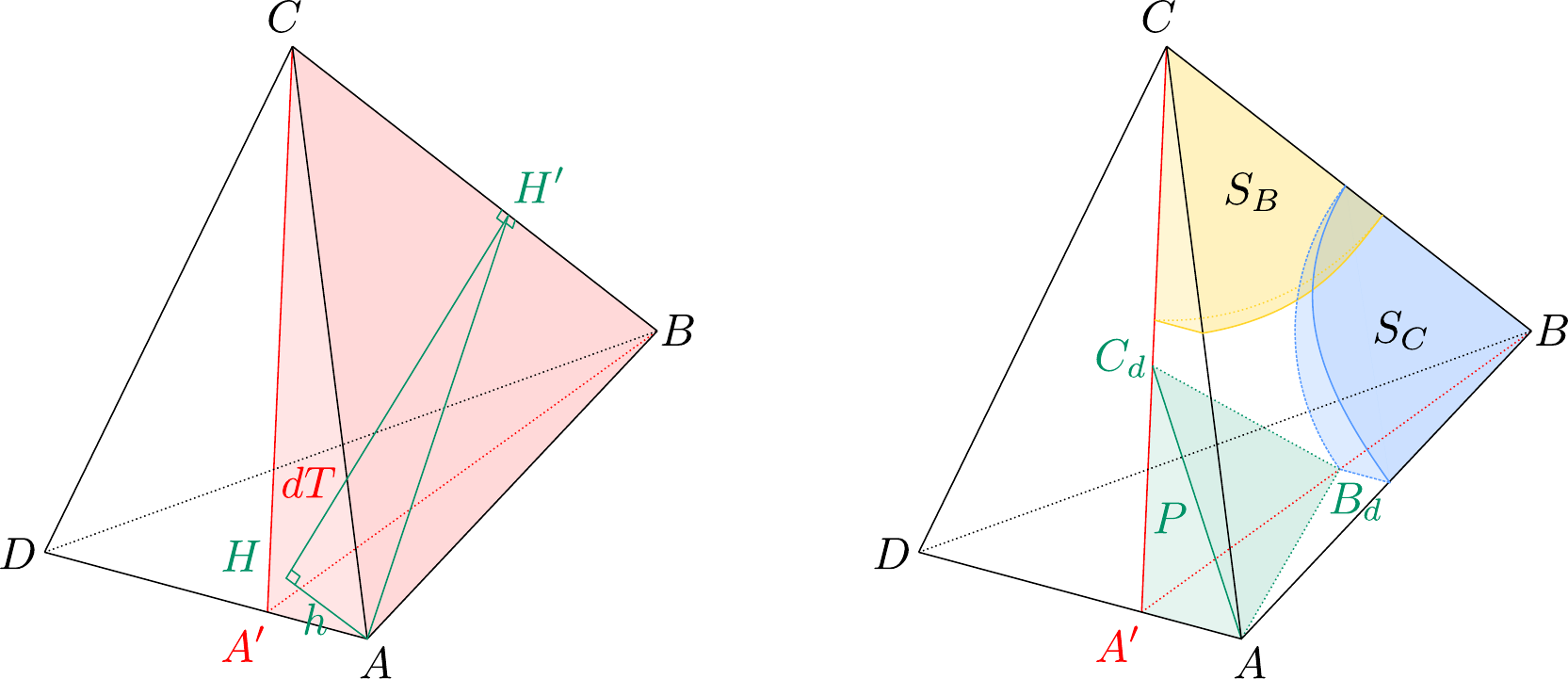}
\caption{
Consider a tetrahedron $T=ABCD$, where the sphere centered in $A$ is slid towards $D$ and becomes centered in $A'$.
Denote by $dT$ the tetrahedron $AA'BC$,
by $H$ the orthogonal projection of $A$ onto $A'BC$,
by $H'$ the orthogonal projection of $H$ onto $BC$
and by $h$ the length $|AH|$ of the edge $AH$ (left).
For $X\in\{B,C\}$, let $S_X$ denote the volume of the intersection of $dT$ with the inflated sphere of center $X$ and radius is $dr_X$; let also $X_d$ denote the point on the edge $A'X$ of $dT$ which is at distance $r_A+(1-d)\max(r_B,r_C)$ from $A'$ (right).
}
\label{fig:compression}
\end{figure}

\begin{proof}
Note: Fig.~\ref{fig:compression} depicts all the notations introduced in this proof.
Consider a tetrahedron $T$ with vertices $A$, $B$, $C$ and $D$.
Let $r_X$ denote the radius of the sphere centered in the vertex $X$.
Assume that the sphere centered in $A$ can be slid towards $D$ and denote by $A'$ its new center -- the edge DA' is thus tight.
Let also $T'$ denote the tetrahedron after the sliding.
We shall prove $\Gamma_{\delta^*_t}(T')>\Gamma_{\delta^*_t}(T)$, where $t$ is the type of $T$.

The sliding has contrasting effects on the compression.
On the one hand, it increases it because the volume of the tetrahedron decreases and the solid angle in $A$ increases.
On the other hand, it decreases the compression because the solid angles in both $B$ and $C$ decrease.
That the $\delta$-compression, overall, increases will follow from
\begin{equation}
\label{eq:sliding1}
\delta\textrm{vol}(dT)\geq \frac{r_B^3}{3}d\widehat{B}+\frac{r_C^3}{3}d\widehat{C},
\end{equation}
where $dT$ denotes the tetrahedron $T\backslash T'$ and $d\widehat{X}$, for $X\in\{B,C\}$ denotes the solid angle loss in $X$.
In particular, we are not going to exploit the solid angle gain in $A$.

The difficulty in comparing the volume and solid angle losses lies in the fact that their formulas are very different.
The idea implemented in \cite{Hal97} is to inflate the spheres centered in $B$ and $C$ by a suitable factor $d>1$ and to bound the volume loss from below by the volume of a region $P\subset T'\backslash T$ which does not intersect the inflated spheres.
The point of inflating spheres is that the resulting volume gain can be expressed as a function of solid angle losses.

Let us formalize this.
Let $H$ denote the foot of the altitude from $A$ of $dT$, $h$ the length of this altitude and $H'$ the orthogonal projection of $H$ onto the line $(BC)$.
For $X\in\{B,C\}$, let $S_X$ denote the volume of the intersection of $dT$ with the inflated sphere of center $X$ and radius is $dr_X$.
Let also $X_d$ denote the point on the edge $A'X$ of $dT$ which is at distance $r_A+(1-d)\max(r_B,r_C)$ from $A'$.
For $d>1$, the length $|A'X_d|$ of the edge $A'X_d$ satisfies
\begin{eqnarray*}
|A'X_d|
&=&r_A+(1-d)\max(r_B,r_C),\\
&=&r_A+\min((1-d)r_B,(1-d)r_C),\\
&=&\min(r_A+r_B-dr_B,r_A+r_C-dr_C),\\
&\leq&\min(|A'B|-dr_B,|A'C|-dr_C).
\end{eqnarray*}
That is, both $B_d$ and $C_d$ lie on a sphere of center $A'$ which intersects neither $S_B$ nor $S_C$.
Hence, the tetrahedron $P$ with vertices $A$, $A'$, $B_d$ and $C_d$ overlaps neither $S_B$ nor $S_C$.
The main point now will be to show
\begin{equation}
\label{eq:sliding2}
\textrm{vol}(P)\geq \textrm{vol}(S_B\cap S_C).
\end{equation}
Indeed, this will yield
\[
\textrm{vol}(dT) \geq 
\underbrace{\textrm{vol}(S_B)+\textrm{vol}(S_C)}_{=d^3\left(\frac{r_B^3}{3}d\widehat{B}+\frac{r_C^3}{3}d\widehat{C}\right)}+\underbrace{\textrm{vol}(P)-\textrm{vol}(S_B\cap S_C)}_{\geq 0}.
\]
This yields \eqref{eq:sliding1} for any $d$ which satisfies
\begin{equation}
\label{eq:sliding3}
\delta d^3\geq 1.
\end{equation}
To show \eqref{eq:sliding2}, we will first bound $\textrm{vol}(P)$ from below, then bound $\textrm{vol}(S_B\cap S_C)$ from above and finally combine these bounds to get \eqref{eq:sliding2} with a value $d$ which satisfies \eqref{eq:sliding3}.
This will yield \eqref{eq:sliding1} and thus conclude the proof of the claimed result.

\paragraph{Minoration of $\textrm{vol}(P)$.}
We have
\[
\textrm{vol}(P)=h\times\frac{1}{3}\textrm{area}(A'B_dC_d)=h\times \frac{1}{6}\sin\widehat{BA'C}\times |A'B_d|\times |A'C_d|.
\]
The radius of the support sphere of $T$ being at most $r$, the radius of the sphere circumscribed to $T$ is at most $1+r$.
This also bounds from above the radius of the circle circumscribed to the triangle $BA'C$, so that the law of sines yields
\[
\frac{|BC|}{\sin\widehat{BA'C}}\leq 2(1+r).
\]
Since $|BC|\geq r_B+r_C$, this ensures
\[
\sin\widehat{BA'C}\geq \frac{r_B+r_C}{2(1+r)}.
\]
Finally, we get
\begin{equation}
\label{eq:sliding4}
\textrm{vol}(P)\geq h\frac{(r_B+r_C)(r_A+(1-d)\max(r_B,r_C))^2}{12(1+r)}.
\end{equation}

\paragraph{Majoration of $\textrm{vol}(S_B\cap S_C)$.}
If $S_B\cap S_C=\emptyset$, then \eqref{eq:sliding2} is proven.
Assume $S_B\cap S_C\neq\emptyset$, that is, $|BC|\leq d(r_B+r_C)$.
The intersection of $S_B$ and $S_C$ is a sort of (possibly asymmetrical) lens included in a spherical cap of radius $R:=d\min(r_B,r_C)$ and height
\[
x:=d(r_B+r_C)-|BC|\leq (d-1)(r_B+r_C).
\]
Such a spherical cap has volume
\[
\frac{\pi}{3}(3R-x)x^2.
\]
But $S_B\cap S_C$ is also delimited by the faces $ABC$ and $A'BC$, so its volume is actually bounded by the volume of a ``pie slice'' of angle $\widehat{BC}_{AA'}$ of the above spherical cap, where $\widehat{BC}_{AA'}$ denotes the dihedral angle between the faces $ABC$ and $A'BC$.
This yields
\begin{eqnarray*}
\textrm{vol}(S_B\cap S_C)
&=& \frac{\pi}{3}(\underbrace{3d\min(r_B,r_C)-x}_{\leq 3d\min(r_B,r_C)})x^2\times\frac{\widehat{BC}_{AA'}}{2\pi}\\
&\leq& \frac{1}{2}d\min(r_B,r_C)(d-1)^2(r_B+r_C)^2\times\widehat{BC}_{AA'}.
\end{eqnarray*}
Further, one has
\[
\widehat{BC}_{AA'}=\widehat{AH'H}\leq\tan\widehat{AH'H}=\frac{h}{|HH'|}=\frac{h}{\sqrt{|AH'|^2-h^2}}.
\]
Now, according to Prop.~\ref{prop:FM_faces} (the only place it is used), $ABC$ is an FM-triangle, hence the sphere in $A$ does not intersect the edge $BC$ (as proven in \cite{FM58}).
Similarly, the spheres in $B$ and $C$ do not intersect the edge $AC$ and $AB$, respectively.
This ensures either $|AH'|\geq r_A$ if $H'$ is between B and C,
or $|AH'|\geq r_C$ if $H'$ is closer to $C$ than to $B$
or $|AH'|\geq r_B$ if $H'$ is closer to $B$ than to $C$.
In any case:
\[
|AH'|\geq\min(r_A,r_B,r_C).
\]
Finally, we get
\begin{equation}
\label{eq:sliding5}
\textrm{vol}(S_B\cap S_C)\leq h\frac{\min(r_B,r_C)(d-1)^2(r_B+r_C)^2}{2\sqrt{\min(r_A,r_B,r_C)^2-h^2}}.
\end{equation}

\paragraph{Combining the minoration and the majoration.}
Combining Inequalities~\eqref{eq:sliding4} and \eqref{eq:sliding5} we get that the inequality stated in the lemma ensures, for $h$ small enough, Inequality \eqref{eq:sliding2}.
The $\delta$-compression then does not decrease.
\end{proof}

\subsection{Sliding to the first contact}

As mentioned, besides compression, the problem with sliding is that the radius of the support sphere may increase (or such a sphere may even not exist), so we might leave the space of FM-tetrahedra.
For this reason, we use sliding in a specific way.
Namely, we slide equally the four spheres of the considered FM-tetrahedron towards the center of its support sphere (by {\em equally} we mean that the displacement of each vertex has the same magnitude).
The fact that the obtained tetrahedron is FM is then straightforward: its support sphere is still centered on the same point and its radius is decreased by the magnitude of the vertex displacement (hence is at most $r$).
We can do this until a first contact between two spheres is reached (Fig.~\ref{fig:first_contact}).

\begin{figure}[htbp]
\centering
\includegraphics[width=0.9\textwidth]{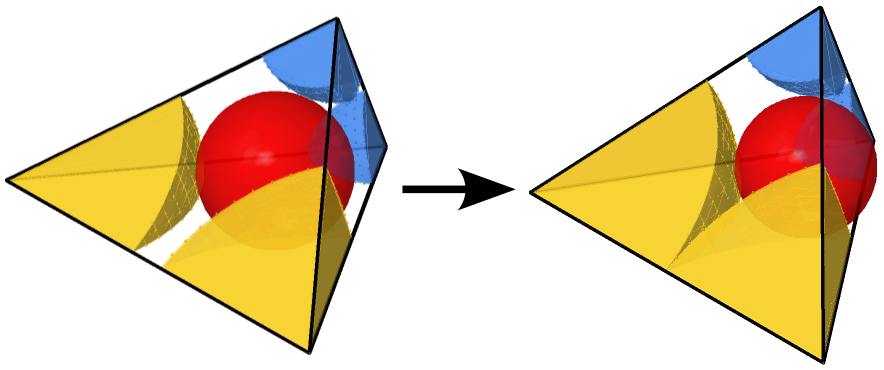}
\caption{
Sliding spheres towards the support sphere (in red) until a contact between spheres (here the two large ones) is reached (from left to right).
Since the support sphere shrinks (its center does not move), we are ensured not to leave the space of FM-tetrahedra.
}
\label{fig:first_contact}
\end{figure}

The following proposition shows that, starting from a tetrahedron less dense than the claimed maximum density $\delta_t^*$, this transformation does not allow us to exceed this density:

\begin{proposition}
\label{prop:first_contact}
Let $T=ABCD$ be an FM-tetrahedron of type $t$ and of density at most $\delta_t^*$.
Then, sliding equally the spheres of $T$ towards the center $O$ of the support sphere of $T$ yields an FM-tetrahedron of density at most $\delta_t^*$.
\end{proposition}

\begin{proof}
By definition of the compression, having density at most $\delta_t^*$ is equivalent to having a non-negative $\delta_t^*$-compression.
Since it holds for $T$, it suffices to prove that the $\delta_t^*$-compression does not decrease during the sliding: this will ensure that it holds for the obtained tetrahedron.
Note that we do not prove that the density increases (though experiments on random FM-tetrahedra suggest it is the case).

This is proven by applying Lemma~\ref{lem:sliding} to each of the four subtetrahedra defined by the center $O$ of the support sphere and each of the four faces of $T$ (indeed, the sphere towards which we slide does not play any role in Lemma~\ref{lem:sliding}).
For that, we just check by computation that the inequality in Lemma~\ref{lem:sliding} holds.
This computation can be found in the file \verb+sliding.sage+.
\end{proof}

In particular, if some FM-tetrahedron of type $t$ has density exceeding $\delta_t^*$, then so does the FM-tetrahedron obtained by sliding the spheres equally as much as possible towards the support sphere.
Hence, in order to verify that every FM-tetrahedron of type $t$ has density at most $\delta_t^*$, it suffices to check it for the FM-tetrahedra with two spheres in contact, i.e., with a tight edge.
By considering the combinations of sphere radii and contact pairs, we get, up to symmetry, the $9$ cases depicted in Fig.~\ref{fig:one_contact}.

\begin{figure}[htbp]
\includegraphics[width=\textwidth]{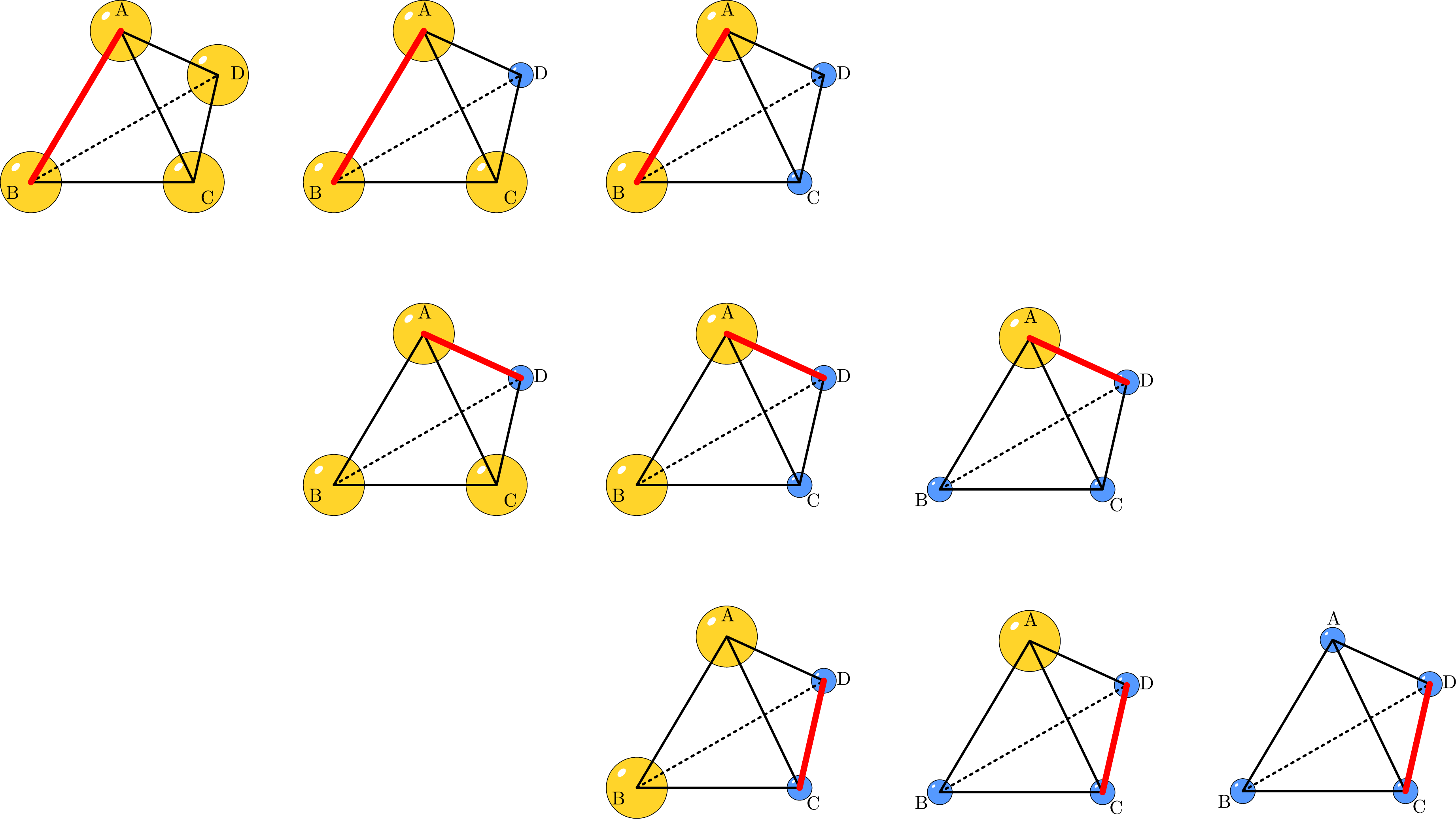}
\caption{
There are $5$ types of tetrahedra (one per column).
The radii of the two touching spheres (red edge) can be $1-1$, $1-r$ or $r-r$ (one case per line).
This yields, up to symmetry, these $9$ cases.
}
\label{fig:one_contact}
\end{figure}

It is tempting to try to slide more spheres to obtain a second contact, or even more.
However, there are cases where sliding two spheres of a tetrahedron towards each other increases the radius of the support sphere beyond $r$, i.e. leaves the space of FM-tetrahedra (similarly as in Fig.~\ref{fig:homothety}).
Of course, reaching the limit $R=r$ or a zero discriminant of the polynomial $P$ (Prop.~\ref{prop:radius}) also reduces the dimension (the length of a given edge can be deduced from $R$ and the lengths of the other edges).
Unfortunately, in practice, the complexity of the involved formulas makes the computation even more costly than stopping at a single contact.
We have therefore abandoned this approach.

\section{Implementation}
\label{sec:computer}

\subsection{Solving a quadratic equation}

Prop~\ref{prop:radius} shows that computing the radius $R$ of the support sphere of a tetrahedron amounts to finding the roots of a quadratic polynomial.
This is usually simple, but here we have to be a bit more careful, because the coefficients are not real numbers but intervals.
Let us detail this.
Assume we have to find the real roots of a quadratic polynomial $AX^2+BX+C$, where $A$, $B$ and $C$ are intervals.

First, if the discriminant $\Delta:=B^2-4AC$ is negative (that is, it is an interval included in $(-\infty,0)$), then there is no real root.
We thus set
\[
\Delta=(B^2-4AC)\cap [0,\infty)
\]
and assume it is not empty.
The polynomial has now two real roots:
\[
X_1:=\frac{-B+\sqrt{\Delta}}{2A}
\quad\textrm{and}\quad
X_2:=\frac{-B-\sqrt{\Delta}}{2A}.
\]
The narrower the intervals $A$, $B$ and $C$ are, the better is the precision on the roots, except if $A$ contains $0$.
In that case, the roots are both $(-\infty,\infty)$, no matter how narrow $A$ is.
This happens for blocks which contain a tetrahedron such that $A$ is equal to zero\footnote{Since $A=0$ defines an algebraic variety which intersects the space of FM-tetrahedra (except for type 1111), there is a continuum of such cases.
An explicit example is the type 111r tetrahedron with edge lengths $(2,2,1+r,2\sqrt{2-r},1+r,1+r)$.}.
However, for $A=0$, the polynomial is $BX+C$ and has a well defined root, namely $-C/B$.
Actually, the root $X_1$ tends towards $-C/B$ when $A$ tends towards $0$, while $X_2$ tends towards plus or minus infinity.
We therefore want to compute
\begin{enumerate}
\item an interval for $X_1$ which becomes narrower around $-C/B$ when $A$ tends towards $0$;
\item a lower bound on $X_2\cap[0,\infty)$ which tends towards infinity when $A$ tends towards $0$.
\end{enumerate}
Let $x:=4AC/B^2$.
Since $\Delta\geq 0$, we have $x\leq 1$.
Let us first assume $0\notin B$ (we will come back to this case at the end).
We can then assume $B>0$, otherwise we consider the polynomial $-AX^2-BX-C$ which has the same roots.
On the one hand,
\begin{eqnarray*}
X_1&=&\frac{-B+B\sqrt{1-x}}{2A}\\
&=&-\frac{C}{B}\frac{B^2}{2AC}(1-\sqrt{1-x})\\
&=&-\frac{C}{B}\frac{2}{x}(1-\sqrt{1-x})\\
&=&-\frac{C}{B}\frac{2}{1+\sqrt{1-x}}.
\end{eqnarray*}
With $f(x):=\tfrac{2}{1+\sqrt{1-x}}$:
\begin{equation}
\label{eq:X1}
X_1=-\frac{C}{B}f(x).
\end{equation}
Since $f$ is continuous over $(-\infty,1]$ and $f(0)=1$, this yields the wanted interval for $X_1$.
On the other hand,
\[
X_2
=\frac{-B-B\sqrt{1-x}}{2A}
=-\frac{B}{A}\frac{1+\sqrt{1-x}}{2}
=-\frac{B}{A}\frac{1}{f(x)}.
\]
We thus get for $\tfrac{1}{X_2}$ the interval
\begin{equation}
\label{eq:X2}
\frac{1}{X_2}=-\frac{A}{B}f(x).
\end{equation}
The positive values of $X_2$ (if any) are then bounded from below by the inverse of the right endpoint of this interval.
Since this quantity tends towards infinity when $A$ tends towards $0$, this yields the wanted lower bound on the positive values of $X_2$.

The above reasoning holds under the assumption $0\notin B$.
But if $0\in B$, then \eqref{eq:X1} yields $X_1=(-\infty,\infty)$, which is clearly an interval containing all the real roots of the polynomial.
Computing $X_1$ and $X_2$ according to \eqref{eq:X1} and \eqref{eq:X2} thus still yields a valid interval.

Now, the narrower the intervals $A$, $B$ and $C$, the better the precision on the roots, except in cases where there exists a tetrahedron for which both $A$ and $B$ are zero (and thus $C$ as well).
Such tetrahedra do exist\footnote{Recall the flat tetrahedron mentioned after Definition~\ref{def:support_sphere}.} but since the recursive checking of the blocks eventually terminated, they were ruled out for other reasons, e.g. a density lower than the claimed bound.

These computations are implemented in the file \verb+trinome.cpp+.

\subsection{Formula optimization}
\label{sec:formulas_opt}

One of the peculiarities of interval arithmetic that can confuse mathematicians is the fact that the result of evaluating an expression depends on the expression itself.
In other words, among many different but equivalent expressions, some may be better than other.
To paraphrase Orwell: all expressions are equal, but some expressions are more equal than others.
This is known as the {\em dependency problem} and it is a major obstacle to the application of interval arithmetic. 

For example, the two following expressions are mathematically equivalent:
\[
xy+3x+2y
\quad\textrm{and}\quad
(x+2)(y+3)-6.
\]
However, for the intervals $x=[1,2]$ and $y=[-3,-2]$, the first expression yields the interval $[-9,0]$ while the second one yields $[-6,-2]$:
\begin{verbatim}
sage: x=RealIntervalField()(1,2)
sage: y=RealIntervalField()(-3,-2)
sage: (x*y+3*x+2*y).endpoints()
(-9.00000000000000, -0.000000000000000)
sage: ((x+2)*(y+3)-6).endpoints()
(-6.00000000000000, -2.00000000000000)
\end{verbatim}

This issue is important when checking an inequality over an interval, because it means that, depending on the expression, the recursive interval subdivision will go more or less far, which essentially determines computing time.
A rule of thumb in interval arithmetic is to minimize the number of variable occurrences.
This is however not an absolute rule and, moreover, it is intractable \cite{BU11}.
In the file \verb+radius_optimization.sage+, we define a heuristic in order to simplify the expressions of the coefficients and the discriminant of the quadratic polynomials which gives the radius of the support sphere of a tetrahedron.
These expressions are polynomials and the heuristic is a sort of ``greedy partial factorization'' which proceeds on an expression $e$ recursively as follows:
\begin{itemize}
\item if $e$ has degree $1$, return $e$;
\item if $e=f\times g$, apply recursively on $f$ and $g$;
\item otherwise, let $x^d$ be a monomial among those with the highest number of occurrences in $e$; write $e=x^df+g$; apply recursively on $f$ and $g$.
\end{itemize}
For example, this automatically rewrites the following expression which has $23$ variables occurrences
\[
x_{1} x_{2} x_{3} - x_{2}^{2} x_{3} - x_{2} x_{3}^{2} - x_{1}^{2} x_{4} + x_{1} x_{2} x_{4} + x_{1} x_{3} x_{4} + x_{2} x_{3} x_{4} - x_{1} x_{4}^{2} - x_{1} x_{2} x_{5}
\]
into this one, which has $14$ variables occurrences:
\[
-{\left(x_{2} + x_{3} - x_{4}\right)} x_{2} x_{3} + {\left(x_{2} {\left(x_{3} + x_{4} - x_{5}\right)} - {\left(x_{1} - x_{3} + x_{4}\right)} x_{4}\right)} x_{1}.
\]
This is not always optimal.
For example, the above formula can actually be written with only 11 occurrences:
$$
(x_1-x_2-x_3+x_4)(x_2x_3-x_1x_4)-x_1x_2x_5.
$$
In \verb+radius_optimization.sage+, the coefficients of $P$ (as well as its discriminant) totalize $18343$ variable occurrences over the $9$ considered cases.
This has been reduced to $1051$ by the above heuristic.
The simplified expressions are those appearing in the file file \verb+radius.cpp+ 

Let us now consider the volume, which is crucial in the computation of the density.
The already mentioned Cayley-Menger determinant gives the square of the volume multiplied by $288$.
It is a polynomial with $60$ variables occurrences.
The above heuristic reduces it to $30$ occurrences, and with further manual simplification, we managed to reduce it to $21$ (this is the formula used in the file \verb+routines.cpp+, where the $x_i$'s are squared edge lengths):\\

\noindent\resizebox{\textwidth}{!}{$
\frac{{\left({\left(x_{0} + x_{1} - x_{3}\right)}^{2} - 4 \, x_{0} x_{1}\right)} {\left({\left(x_{0} + x_{2} - x_{4}\right)}^{2} - 4 \, x_{0} x_{2}\right)}-{\left({\left(x_{0} + x_{1} - x_{3}\right)} {\left(x_{0} + x_{2} - x_{4}\right)} - 2 \, x_{0} {\left(x_{1} + x_{2} - x_{5}\right)}\right)}^{2}}{2 \, x_{0}}.$}\\

Actually, we combine this formula with another one, obtained by the mean value theorem.
This is based on the same idea used when solving a quadratic equation when the leading coefficient is an interval which contains $0$ or near the optimal tetrahedra.
For a differentiable function $f$ of a real variable $x$, we can write $f(x)=f(x_0)+(x-x_0)f'(z)$ where $z$ is between $x$ and $x_0$.
With $x$ being an interval and $x_0$ a point in $x$, this yields
\[
f(x)\subset f(x_0)+(x-x_0)f'(x).
\]
When $f$ has a complicated expression, we get an even worse expression for $f'$.
Hence, typically, the interval $f'(x)$ is larger than $f(x)$, which is undesirable.
However, we multiply $f'(x)$ by $x-x_0$, whose width is the width of $x$, that becomes smaller and smaller when we recursively divide it ($f(x_0)$ is not a problem because $x_0$ is a point, in practice a very thin interval, so $f(x_0)$ is also very thin).
Thus, for sufficiently thin $x$, using a first order formula can yield a thinner interval.
We implemented this in file \verb+routines_order1.cpp+.
Since it is not clear in advance which method yields the best result (it depends on the considered block) we compute both intervals and take the intersection, which, as usual, contain the mathematically exact result.

\subsection{Final assault}
\label{sec:final}

The file \verb+check.cpp+ verifies each of the $9$ cases depicted in Fig.~\ref{fig:one_contact} using the strategy described in Section~\ref{sec:strategy}.
The space of FM-tetrahedra is divided in millions of {\bf blocks} -- a set of FM-tetrahedra whose edge lengths are intervals.
Every block is discarded if
\begin{itemize}
\item it does not intersect the space of FM-tetrahedra;
\item it lies in the neighborhoods of extrema obtained in Section~\ref{sec:epsilon};
\item it contains only tetrahedra whose density is less than the bounds claimed in Theorem~\ref{th:main}.
\end{itemize}
All the blocks turned out to be eventually discarded by the program.
Together with the analysis given in Section~\ref{sec:epsilon}, this proves Theorem~\ref{th:main}.
Table~\ref{tab:stats} provides, for each of the $9$ cases depicted in Fig.~\ref{fig:one_contact}, the number of blocks discarded during the recursive subdivision.
In total, we have examined around 2.5 billions of blocks (most of them for the type $rrrr$ case).
This required around 240 hours of CPU time on our laptop (11th Gen Intel Core i5-1145G7).
In practice, the computation was parallelized: a pool of threads handles different blocks and new threads are created when subdividing a block (as long as it does not exceed a specified maximum number of threads); it ran for around $30$ hours on our laptop.

\begin{table}
\centering
\begin{tabular}{|c|ccccc|}
\hline
&1111 & 111r & 11rr & 1rrr & rrrr\\
\hline
11 & $16\,032$ & $14\,964\,193$ & $\hphantom{0}14\,611\,062$ & &\\
1r & & $75\,943\,637$ & $112\,578\,247$ & $35\,775\,615$ &\\
rr & & & $204\,225\,779$ & $52\,226\,020$ & $2\,040\,485\,620$\\
\hline
\end{tabular}
\caption{
Number of blocks discarded while checking the $9$ cases depicted in Fig.~\ref{fig:one_contact}.
The columns correspond to the types of the tetrahedra and the rows to the tight edge.
The computation time is roughly proportional to the number of discarded blocks.
}
\label{tab:stats}
\end{table}

\section*{Acknowledgements}

We would like to express our thanks to Nathalie Revol for her advice on interval arithmetic (especially techniques at order 1), and to Michael Rao for his advice on programming in C++ (especially parallelization with a thread pool).
We also warmly thank the anonymous referee, who notably pointed us toward the code \cite{Lei25}, which makes it easy to reproduce the results of \cite{Lov14} -- results that we now realize we have hitherto examined too superficially.

\appendix

\section{Code overview}
\label{sec:code}

The computer plays a central role in proving the main theorem of this article.
The code is divided into 10 programs that can be found there:
\begin{center}
\url{https://github.com/fernique/florian3D}
\end{center}
We have tried to keep the code as simple and concise as possible, to make it as easy to read as possible (in total, there is only slightly more than $1000$ lines of code -- although this is obviously only a very crude measure of complexity).
Here, we list all these programs in alphabetical order and recall their respective roles (they are already all mentioned in the body of the article).

\begin{itemize}

\item \verb+appendix_B.sage+ (77 lines).
This SageMath file provides the computations used in Appendix~\ref{sec:proof_FM_faces}.

\item \verb+check.cpp+ (246 lines).
This is the main C++ file (hence the one to be compiled).
It includes the four C++ files listed after.
It checks that all the FM-tetrahedra not in the neighborhood of the densest ones are less dense.
The first line has to be changed to choose which of the $9$ cases it checks.
The constant \verb+maxThreads+ can be modified according to the number of cores of the computer the program will run on.
There are only four functions (besides the main):
\begin{itemize}
\item \verb+type+: determines the ``type'' of a block B of tetrahedra, that is an integer equal to $0$ for tetrahedra near optimal ones (case solved in \verb+local.sage+), $1$ if it is not an FM-tetrahedron, $2$ if it is ``hollow'' that is less dense that the claimed maximal density and, last but not least, $-1$ if the block is still too large to determine its type and needs to be further divided;
\item \verb+halve_block+: halves a block of tetrahedra along its widest side;
\item \verb+explore+: recursively explores the space of FM-tetrahedra (using a pool of threads);
\item \verb+initial_block+: defines the block on which the exploration is started, depending on the case to be checked.
\end{itemize}

\item \verb+local.sage+ (335 lines).
This SageMath file is devoted to the results explained in Section~\ref{sec:epsilon}.
Namely, we provide an explicit neighborhood of the tetrahedra claimed to be the densest in Theorem~\ref{th:main} over which these tetrahedra are indeed the densest ones (among FM-tetrahedra).

\item \verb+radius.cpp+ (84 lines).
This completes \verb+routines.cpp+ with the computation of the radius of the support sphere.
The expressions of the coefficients of the polynomial P of Proposition~\ref{prop:radius} are case-dependent (which makes the file not very pleasant to read, but greatly enhances the computations).

\item \verb+radius.sage+ (58 lines).
In this SageMath file, we derive the polynomial $P$ which is quadratic in the radius $R$ of the support sphere (Proposition~\ref{prop:radius}).

\item \verb+radius_optimization.sage+ (192 lines).
This SageMath file implements the heuristics described in Section~\ref{sec:formulas_opt} to optimize the coefficients of the polynomial $P$ in each of the $9$ considered cases (Fig.~\ref{fig:one_contact}).
It also outputs the optimized formula in a format that can be used in C++.

\item \verb+routines.cpp+ (115 lines).
This file contains the functions to test the validity of a block of tetrahedra, compute its volume and its density.

\item \verb+routines_order1.cpp+ (68 lines).
This completes the file \verb+routines.cpp+ with the computation at order $1$ of the volume, as explained at the end of Section~\ref{sec:formulas_opt}.

\item \verb+sliding.sage+ (46 lines).
This SageMath file checks that the inequality stated in Lemma~\ref{lem:sliding} holds for every case used in the proof of Proposition~\ref{prop:first_contact}.

\item \verb+trinome.cpp+ (45 lines).
This routine is used in \verb+radius.cpp+ to compute the roots of a quadratic polynomial (namely, the radius of the support sphere).

\end{itemize}

\section{Proof of Proposition~\ref{prop:FM_faces}}
\label{sec:proof_FM_faces}

We shall need some lemmas.
Every non-trivial computation appearing in the proofs of this section is detailed in the file \verb+appendix_B.sage+.

\begin{lemma}
\label{lem:quadradic_poly}
In $\mathbb{R}^3$, consider three interior disjoint spheres of positive radii $r_A$, $r_B$ and $r_C$ respectively centered on the vertices A, B and C of a triangle.
Consider a (fourth) sphere which is interior disjoint and tangent to each of the three previous spheres.
Let $R$ denote its radius.
Then, $R$ is a root of a quadratic polynomial $P=aR^2+bR+c(z)$ whose coefficients depend on the edge length of the triangle ABC and the radii of the spheres centered on the vertices of this triangle, and $c$ depends also on the distance $z$ of the center of the fourth sphere to the plane that contains the triangle ABC.
\end{lemma}

\begin{proof}
Denote by ab, ac and bc the length of the edges AB, AC and AD.
{\em W.l.o.g.}, A, B and C have coordinates $(0,0,0)$, $(ab,0,0)$ and $(x_c,y_c,0)$, with $y_c>0$. 
Let $(x,y,z)$ be the coordinates of the center of the sphere which is tangent to the three previous spheres (so that $z>0$ is the distance of the center of the fourth sphere to plane that contains the triangle ABC).
Let R be its radius.

We write the three equations given by the contact of the fourth sphere to the three first spheres, as well as the two equations which relates $x_c$ and $y_c$ to the edge lengths $ac$ and $bc$.
This system of equations is easily solved using computer algebra.
It yields the polynomial $P$ (its coefficients $a$, $b$ and $c$ have up to $7$ variables and degree $6$, with a total of $76$ monomials: this is why we do not write them here).
\end{proof}

\begin{lemma}
\label{lem:negative_C}
The coefficient $c(z)$ of the polynomial $P$ is negative.
\end{lemma}

\begin{proof}
A computation yields:
\[
c'(z)=-2(ab+ac+bc)(ab+ac-bc)(ab+bc-ac)(ac+bc-ab)z.
\]
Since $ab$, $ac$ and $bc$ are edge lengths of a triangle, each factor of $c'(z)$ is non-negative, whence $c'(z)\leq 0$.
Now, let us prove $c(0)<0$.
A computation shows that $c(0)$ is the sum of the five following terms:
\begin{eqnarray*}
c_0&=&(r_A^2+r_B^2-ab^2)(r_A^2+r_C^2-ac^2)(r_B^2+r_C^2-bc^2),\\
c_1&=&-bc^4r_A^2 + 2bc^2r_A^2r_B^2 - r_A^2r_B^4 + 2bc^2r_A^2r_C^2 - r_A^2r_C^4,\\
c_2&=&-ac^4r_B^2 + 2ac^2r_B^2r_C^2 - r_B^2r_C^4 + 2ac^2r_B^2r_A^2 - r_B^2r_A^4,\\
c_3&=&-ab^4r_C^2 + 2ab^2r_C^2r_A^2 - r_C^2r_A^4 + 2ab^2r_C^2r_B^2 - r_C^2r_B^4,\\
c_4&=&-2r_A^2r_B^2r_C^2.
\end{eqnarray*}
Since the spheres are interior disjoint, one has $ab\geq r_A+r_B$.
Hence
\[
ab^2\geq (r_A+r_B)^2=r_A^2+r_B^2+2r_Ar_B> r_A^2+r_B^2.
\]
This ensures $c_0<0$. 
Clearly, $c_4<0$.
Now, we claim that $c_1$, $c_2$ and $c_3$ are also negative.
Remark that $c_1$, $c_2$ and $c_3$ are equal up to a circular permutation of the vertices A, B, C.
It thus suffices to consider one of them, say, $c_1$.
We write
\begin{eqnarray*}
c_1
&=&r_A^2(-bc^4 + 2bc^2r_B^2 - r_B^4 + 2bc^2r_C^2 - r_C^4)\\
&=&r_A^2(bc^4-(bc^2-r_B^2)^2-(bc^2-r_C^2)^2).
\end{eqnarray*}
We thus have to prove that, for any given positive $a$ and $b$, the function
\[
f(x)=x^4-(x^2-a^2)^2-(x^2-b^2)^2
\]
is negative for every $x\geq a+b$.
This is established by computing
\[
f'(x)=4x(a^2+b^2-x^2)<0
\quad\textrm{and}\quad
f(a+b)=-2a^2b^2<0.
\]
Therefore, $c(0)$ being a sum of five negative terms, it is negative.
With $c'(z)<0$ this yields the claimed result.
\end{proof}

\begin{remark}
The key point of the above proof is the decomposition of $c(0)$ in a sum of five suitable terms.
Let us explain how this decomposition was obtained, which may be more interesting than the proof itself.

First, we defined $c_0$ by guessing, trying to find a factorized expression whose sign was easy to determine and which captured as many monomials of $c(0)$ as possible.

Then, we noticed that subtracting the single term $c_4$ to $c(0)-c_0$ yields $15$ terms which form $5$ groups of $3$ terms equal up to a circular permutation on A,B,C.
Using the computer, we considered the $3^5$ sums formed by one term from each of these $5$ groups and, for each, we checked on a few thousand randomly generated examples whether it was always negative.
The only sums to pass this test were $c_1$, $c_2$ and $c_3$.
Relatively simple, we were finally able to demonstrate their negativity by hand.
\end{remark}

\begin{lemma}
\label{lem:increasing_radius}
If the polynomial $P$ has a positive root for some $z_0>0$,
then it has at least one positive root for every $z\in[0,z_0]$
and the smallest positive root is non-decreasing in $z$.
\end{lemma}

\begin{proof}
Let $\Delta(z)$ denote the discriminant of $P$.
A computation yields
\[
\Delta(Z)=ez^2+\Delta(0),
\]
where $e$ does not depend on $z$.
More precisely, $\Delta(0)$ is the product of
\[
(ab+ac+bc)(ab+ac-bc)(ab-ac+bc)(ab-ac-bc)
\]
and
\[
(ab+r_A-r_B)(ab-r_A+r_B)(ac+r_A-r_C)(ac-r_A+r_C)(bc+r_B-r_C)(bc-r_B+r_C).
\]
The first term is non-negative because ABC is a triangle.
The second term is positive because spheres do not intersect.
This ensures $\Delta(0)\geq 0$.
Random tries show that $e$ can be positive or negative.
However, for $z=z_0$ we have $\Delta(z_0)=ez_0^2+\Delta(0)\geq 0$ since $P$ has a real root.
This ensures $\Delta(z)\geq 0$ for every $z\in[0,z_0]$, that is, $P$ has real roots over this interval.
These roots are
\[
R_-=\frac{-b-\sqrt{b^2-4ac}}{2a}
\quad\textrm{and}\quad
R_+=\frac{-b+\sqrt{b^2-4ac}}{2a}.
\]
Recall that $c<0$ by Lemma~\ref{lem:negative_C}.
Hence
\begin{itemize}
\item if $a>0$, then $\sqrt{b^2-4ac}>|b|$ and $R_-\leq 0\leq R_+$: there is a unique positive root, $R_+$.
\item if $a<0$, then $\sqrt{b^2-4ac}<|b|$: either $R_+\leq R_-\leq 0$ if $-b/a<0$ or $0\leq R_+\leq R_-$ otherwise.
The first case is excluded because, for $z=z_0$, $P$ would not have a positive root (neither $a$ nor $b$ depends on $z$).
Thus $0\leq R_+\leq R_-$ and the smallest positive root is, again, $R_+$.
\end{itemize}
Now, a computation yields
\[
R_+'(z)=\frac{
(ab + ac + bc)(ab + ac - bc)(ab - ac + bc)(ac + bc -ab)z}{\sqrt{\cdots}}.
\]
The numerator is non-negative for $z\in[0,z_0]$.
The denominator is positive as the square root of some expression (namely a polynomial of degree $10$ with $339$ monomials).
This ensures that the smallest positive root $R_+$ is non-decreasing in $z$.
\end{proof}

\noindent We are now in a position to prove Proposition~\ref{prop:FM_faces}:\\

\begin{proof}
Consider an FM-tetrahedron and a face $ABC$ of it.
By definition, its support sphere is tangent to the spheres in A, B and C.
Its radius $R$ is thus a root $R(z_0)$ of a quadratic polynomial $P$ by Lemma~\ref{lem:quadradic_poly}.
It is the smallest positive root because any larger positive root would correspond to a support sphere whose interior intersects the spheres centered in A, B and C.
Since $R$ is positive, Lemma~\ref{lem:increasing_radius} ensures that we can decrease $z$ from $z_0$ to $0$ to get a smaller sphere of radius $R(0)<r$ whose center is in the face ABC.
Seen in the face ABC, this sphere is exactly a support circle for the triangle ABC, which is thus an FM-triangle as claimed.    
\end{proof}

\bibliographystyle{alpha}
\bibliography{florian3D}

\end{document}